\documentclass[12pt]{article}
\usepackage{enumerate}
\usepackage{amsmath}
\usepackage{amsthm}
\usepackage{amsfonts}
\usepackage{amssymb}
\usepackage{xcolor}
\usepackage[numbers]{natbib}
\usepackage{graphicx,xcolor}
\usepackage{bbm}
\usepackage{pgfplots}
\pgfplotsset{compat=1.13}

\theoremstyle{plain}
\newtheorem{question}{Question}

\newtheorem{problem}[question]{Problem}

\newtheorem{theorem}[question]{Theorem}
\newtheorem{proposition}[question]{Proposition}
\newtheorem{corollary}[question]{Corollary}
\newtheorem{lemma}[question]{Lemma}
\newtheorem{claim}[question]{Claim}
\newtheorem{remark}[question]{Remark}
\theoremstyle{definition}
\newtheorem{definition}[question]{Definition}

\numberwithin{question}{section}
\numberwithin{equation}{section}
\newtheorem*{theorem*}{Theorem}

\title{Long paths and connectivity in $1$-independent random graphs}
\author{A. Nicholas Day\thanks{Institutionen f\"or Matematik och Matematisk Statistik, Ume{\aa} Universitet, 901 87 Ume{\aa}, Sweden. Emails: nicholas.day@umu.se and victor.falgas-ravry@umu.se. Research supported by Swedish Research Council grant 2016-03488.} \and Victor Falgas--Ravry\footnotemark[1] \and Robert Hancock\thanks{Faculty of Informatics, Masaryk University, Botanick\'a 68A, 602 00 Brno, Czech Republic. Email: hancock@fi.muni.cz This work has received funding from the European Research Council (ERC) under the European Union's Horizon 2020 research and innovation programme (grant agreement No 648509) and from the MUNI Award in Science and Humanities of the Grant Agency of Masaryk University. This publication reflects only its authors' view; the European Research Council Executive Agency is not responsible for any use that may be made of the information it contains.}}

\begin{document}
\maketitle
\begin{abstract} 
Given a (possibly infinite) connected graph $G$, a probability measure $\mu$ on the subsets of the edge set of $G$ is said to be $1$\textit{-independent} if events determined by edge sets that are at graph distance at least $1$ apart in $G$  are independent.  Call such a probability measure a \textit{$1$-ipm on $G$}, and denote by $\mathbf{G}_{\mu}$ the associated random spanning subgraph of $G$.

Let $\mathcal{M}_{1,\geqslant p}(G)$ (respectively $\mathcal{M}_{1,\leqslant p}(G)$) denote the collection of $1$-ipms $\mu$ on $G$ for which each edge is included in $\mathbf{G}_{\mu}$ with probability at least $p$ (respectively at most $p$). Let $\mathbb{Z}^2$ denote the square integer lattice. Balister and Bollob\'as raised the question of determining the critical value $p_{\star}=p_{1,c}(\mathbb{Z}^2)$ such that for all $p>p_{\star}$ and all $\mu \in \mathcal{M}_{1,\geqslant p}(\mathbb{Z}^2)$, $\left(\mathbf{\mathbb{Z}^2}\right)_{\mu}$ almost surely contains an infinite component.  This can be thought of as asking for a $1$-independent analogue of the celebrated Harris--Kesten theorem.

In this paper we investigate both this problem and connectivity problems for $1$-ipms more generally. We give two lower bounds on $p_{\star}$ that significantly improve on the previous bounds. Furthermore, motivated by the Russo--Seymour--Welsh lemmas, we define a $1$-independent critical probability for long paths  and determine its value for the line and ladder lattices. Finally, for finite graphs $G$ we study $f_{1,G}(p)$ (respectively $F_{1,G}(p)$), the infimum  (respectively supremum) over all $\mu\in \mathcal{M}_{1,\geqslant p}(G)$ (respectively all $\mu \in \mathcal{M}_{1,\leqslant p}(G)$) of the probability that the random spanning subgraph $\mathbf{G}_{\mu}$ is connected. We determine $f_{1,G}(p)$ and $F_{1,G}(p)$ exactly when $G$ is a path, a complete graph and a cycle of length at most $5$. 

Many new problems arise from our work, which are discussed in the final section of the paper. 
\end{abstract}
\section{Introduction}\label{section: introduction}
\subsection{Bond percolation models, $1$-independence and edge-probability}
Let $G$ be a (possibly infinite) connected graph. Spanning subgraphs of $G$ are called \emph{configurations}. In a configuration $H$, an edge is said to be \emph{open} if it belongs to $H$, and \emph{closed} otherwise. A \emph{bond percolation model} on the host graph $G$ is a probability measure $\mu$ on the spanning subgraphs of $G$, i.e. on the space of configurations. Given such a measure, we denote the corresponding random graph model by $\mathbf{G}_{\mu}$, and refer to it as the $\mu$-random graph or $\mu$-random configuration.

In this paper, we study bond percolation models $\mu$ where the states (open or closed) of edges in subsets $F_1,F_2$ of $E$ in a $\mu$-random configuration are independent provided that the edges in $F_1$ and $F_2$ are `sufficiently far apart'. To make this more precise, we make use of the following definition.
\begin{definition}
	Two edge sets $F_1, F_2\subseteq E$ are $k$-distant if $F_1\cap F_2 =\emptyset$ and the shortest path of $G$  from an edge in $F_1$ to an edge in $F_2$ contains at least $k$ edges.
	A bond percolation model $\mu$ on $G$ is $k$-independent if for any pair $(F_1, F_2)$ of $k$-distant edge sets, the intersections $\mathbf{G}_{\mu}\cap F_1$ and $\mathbf{G}_{\mu}\cap F_2$ are independent random variables. 
\end{definition}
So for example $\mu$ is $0$-independent if each edge of $\mathbf{G}_{\mu}$ is open at random independently of all the others, i.e. $\mu$ can be viewed as a product of Bernoulli measures on the edges of $G$. A well-studied $0$-independent model is the Erd{\H o}s--R\'enyi random graph $\mathbf{G}_{n,p}$, where the host graph is $G=K_n$, the complete graph on $n$ vertices, and where $\mu$, known as the $p$\textit{-random measure}, sets each edge to be open with probability $p$, independently of all the others.

In this paper, we focus on the next strongest notion of independence, namely $1$-independence. Measures that are $1$-independent have the property that events determined by vertex-disjoint edge sets are independent.  For many $1$-independent models, the randomness can be thought to `reside in the vertices'. An important example of a $1$-independent model is that of \emph{site percolation} on the square integer lattice. In this case the host graph is the square integer lattice $\mathbb{Z}^2$ (where two vertices are joined by an edge if they lie at distance $1$ apart), and the measure $\mu=\mu_{\mathrm{site}}(\theta)$ is obtained by switching each vertex of $\mathbb{Z}^2$ \emph{on} at random with probability $\theta$,  independently of all the others, and by setting an edge to be open if and only if both of its endpoints are switched on. Site percolation measures may be defined more generally on any host graph in the natural way. 

Site percolation is an example of a broader class of $1$-independent measures where we independently associate to each vertex $v\in V(G)$ a state $S_v$ at random, and set an edge $uv$ to be open if and only if $f(S_u, S_v)=1$, for some deterministic function $f$ (which may depend on $u$ and $v$). We refer to such measures as \emph{vertex-based measures} (see Section~\ref{section: open problems} for a formal definition). Vertex-based measures on $\mathbb{Z}$ are a generalisation of the well-studied notion of {\it two-block factors}, which are vertex-based measures on $\mathbb{Z}$ in which the vertex states $S_u$ are i.i.d. random variables, see \cite{LiggettSchonmannStacey97} for further details.

An important point to note is that while all $0$-independent bond percolation models are a product of Bernoulli measures on the edges of $G$ (with varying parameters), it is well-known that a graph $G$ may support many $1$-independent measures which cannot be realised as vertex-based measures or as general ``block factors'', see for instance~\cite{AaronsonGilatKeanedeValk89,BurtonGouletMeester93,HolroydLigget16}. In particular for most graphs $G$, it is not feasible to generate or simulate the collection of $1$-independent measures of $G$.

\begin{definition}
	Given a bond percolation model $\mu$ on a host graph $G$, the (lower)-edge-probability of $\mu$ is 
	\[\mathrm{d}(\mu):=\inf_{e\in E(G)} \mu\{e\textrm{ is open} \}.\]
\end{definition}
So for instance a $p$-random measure has edge-probability $p$, while a site percolation measure with parameter $\theta$ has edge-probability $\theta^2$.  The collection of $k$-independent bond percolation models $\mu$ on a graph $G$ with edge-probability  $d(\mu)\geqslant p$ is denoted by $\mathcal{M}_{k,\geqslant p}(G)$.
\begin{remark}\label{remark: random sparsification}
Given a measure $\mu \in \mathcal{M}_{k,\geqslant p}(G)$, we may readily produce a measure $\tilde{\mu}\in \mathcal{M}_{k,\geqslant p}(G)$ such that $\tilde{\mu}(\{e\textrm{ is open} \})=p$ for all $e\in E(G)$ via \emph{random sparsification}: independently delete each edge $e$ of $\mathbf{G}_{\mu}$ with probability $p/\mu(\{e\textrm{ is open} \})\in[0,1]$. The resulting bond percolation model on $G$ is clearly $k$-independent and has the property that each edge is open with probability exactly $p$; the corresponding bond percolation measure $\tilde{\mu}$ thus has the required properties. 
\end{remark}

\subsection{Critical probabilities for percolation and motivation for this paper}
Percolation theory is the study of random subgraphs of infinite graphs. Since its inception in Oxford in the 1950s, it has blossomed into a rich theory and has been the subject of several monographs~\cite{BollobasRiordan06, Grimmett99, MeesterRoy96}.
The central problem in percolation theory is to determine the relationship between edge-probabilities and the existence of infinite connected components in bond percolation models.

In the most fundamental instance of this problem, consider an infinite, locally finite connected graph $G$, and let $\mu$ be a $0$-independent bond percolation model on $G$. We say that \emph{percolation} occurs in a configuration $H$ on $G$ if $H$ contains an infinite connected component of open edges. By Kolmogorov's zero--one law, for $G$ and $\mu$ as above, percolation is a tail event whose $\mu$-probability is either zero or one. This allows one to thus define the \emph{Harris critical probability} $p_{0,c}(G)$ for $0$-independent percolation:
\[p_{0,c}(G):= \inf\Bigl\{p\in [0,1]:  \forall \mu\in \mathcal{M}_{0,\geqslant p}(G), \ \mu(\{\mathrm{percolation}\})=1\Bigr\}.\]

\begin{problem}\label{problem: 0-independent percolation}
	Given an infinite, locally finite connected graph $G$, determine $p_{0,c}(G)$.	 	
\end{problem}
One of the cornerstones of percolation theory --- and indeed one of the triumphs of twentieth century probability theory --- is the Harris--Kesten theorem, which established the value of $p_{0,c}(\mathbb{Z}^2)$ to be $1/2$.
\begin{theorem*}[Harris--Kesten Theorem~\cite{Harris60, Kesten80}]
	Let $\mu$ be the $p$-random measure on $\mathbb{Z}^2$. Then	
	\[\mu(\{\mathrm{percolation}\})=\left\{\begin{array}{ll} 0 & \mathrm{if } \quad p\leqslant \frac{1}{2}\\
	1 & \mathrm{if }\quad  p>\frac{1}{2}. \end{array}\right.\]
\end{theorem*}
In this paper, we focus on the question of what happens to the Harris critical probability in $\mathbb{Z}^2$ if the assumption of $0$-independence is weakened to $k$-independence. In particular, how much can \emph{local} dependencies between the edges postpone the \emph{global} phenomenon of percolation?
\begin{definition}
	Let $G$ be an infinite, locally finite connected graph and let $k\in \mathbb{N}_{0}$. The Harris critical probability for $k$-independent percolation\footnote{As observed in~\cite{FalgasRavry12}, a simple $k$-independent variant of Kolmogorov's zero--one law shows that percolation remains a tail event when we consider $k$-independent models.} in $G$ is defined to be:
	\[p_{k,c}(G):= \inf\Bigl\{p\in [0,1]:  \forall \mu\in \mathcal{M}_{k, \geqslant p}(G), \ \mu(\{\mathrm{percolation}\})=1\Bigr\}.\]
\end{definition}

\begin{problem}\label{problem: Harris critical probabiliy for 1-ipm}
	Determine $p_{1, c}(\mathbb{Z}^2)$.
\end{problem}
Problem~\ref{problem: Harris critical probabiliy for 1-ipm} was proposed by Balister and Bollob\'as~\cite{BalisterBollobas12} in a 2012 paper in which they began a systematic investigation of $1$-independent percolation models. Study of $1$-independent percolation far predates their work (see e.g.~\cite{AaronsonGilatKeanedeValk89,  BalisterBollobasStacey93, BalisterBollobasWalters05, BurtonGouletMeester93, Janson84, LiggettSchonmannStacey97}), however, due to important applications of $1$-independent percolation models.

A standard technique in percolation is \emph{renormalisation}, which entails reducing a $0$-independent model to a $1$-independent one (possibly on a different host graph), trading in some dependency for a boost in edge-probabilities. Renormalisation arguments feature in many proofs in percolation theory; a powerful and particularly effective version of such arguments was developed by Balister, Bollob\'as and Walters~\cite{BalisterBollobasWalters05}.

Their method, which relies on comparisons with $1$-independent models on $\mathbb{Z}^2$ (in almost all cases) and Monte--Carlo simulations to estimate the probabilities of bounded events, has been applied to give rigorous confidence intervals for critical probabilities/intensities in a wide variety of settings: various models of continuum percolation~\cite{BalisterBollobasWalters05, BalisterBollobasWalters09, BalisterBollobas13}, hexagonal circle packings~\cite{BenjaminiStauffer13}, coverage problems \cite{BalisterBollobasSarkarKumar07, HaenggiSarkar13}, stable Poisson matchings~\cite{DeijfenHolroydPeres11, DeijfenHaggstromHolroyd12}, the Divide-and-Colour model~\cite{BalintBeffaraTassion13}, site and bond percolation on the eleven Archimedean lattices~\cite{RiordanWalters07} and for site and bond percolation in the cubic lattice $\mathbb{Z}^3$~\cite{Ball14}. The usefulness of comparison with $1$-independent models and the plethora of applications give strong theoretical motivation for the study of $1$-independent percolation.

From a more practical standpoint, many of the real-world structures motivating the study of percolation theory exhibit short-range interactions and local dependencies. For example a subunit within a polymer will interact and affect the state of nearby subunits, but perhaps not of distant ones.  Similarly, the position or state of an atom within a crystalline network may have a significant influence on nearby atoms, while long-range interactions may be weaker. Within a social network, we would again expect individuals to exert some influence in aesthetic tastes or political opinions, say, on their circle of acquaintance, and also expect that influence to fade once we move outside that circle. This suggests that $k$-independent bond percolation models for $k\geqslant 1$ are as natural an object of study as the more widely studied $0$-independent ones.

Despite the motivation outlined above, $1$-independent models remain poorly understood. To quote Balister and Bollob\'as from their 2012 paper:
``$1$-independent percolation models have become a key tool in establishing bounds on critical probabilities [...]. Given this, it is perhaps surprising that some of the most basic questions about $1$-independent models are open''.  There are in fact some natural explanations for this state of affairs.  As remarked on in the previous subsection, there are \emph{many} very different $1$-independent models with edge-probability $p$, and they tend to be harder to study than $0$-independent ones due to the extra dependencies between edges.  In particular simulations are often of no avail to formulate conjectures or to get an intuition for $1$-independent models in general. Moreover, while the theoretical motivation outlined above is probabilistic in nature, the problem of determining a critical constant like $p_{1,c}(\mathbb{Z}^2)$ is extremal in nature --- one has to determine what the worst possible $1$-independent model is with respect to percolation --- and calls for tools from the separate area of extremal combinatorics.

In this paper, we continue Balister and Bollob\'{a}s's investigation into the many open problems and questions about and on these measures.  Before we present our contributions to the topic, we first recall below previous work on $1$-independent percolation.
\subsection{Previous work on $1$-independent models}
Some general bounds for stochastic domination of $k$-independent models by $0$-independent ones were given by Liggett, Schonmann and Stacey~\cite{LiggettSchonmannStacey97}. Amongst other things, their results implied $p_{1,c}(\mathbb{Z}^2)<1$. Balister, Bollob\'as and Walters~\cite{BalisterBollobasWalters05} improved this upper bound via an elegant renormalisation argument and some computations. They showed that in any $1$-independent bond percolation model on $\mathbb{Z}^2$ with edge-probability at least $0.8639$, the origin has a strictly positive chance of belonging to an infinite open component. This remains to this day the best upper bound on $p_{1,c}(\mathbb{Z}^2)$. In a different direction, Balister and Bollob\'as~\cite{BalisterBollobas12} observed that trivially $p_{1,c}(G)\geqslant \frac{1}{2}$ for any infinite, locally finite connected graph $G$. In the special case of the square integer lattice $\mathbb{Z}^2$, they recalled a simple construction due to Newman which gives
\begin{align}\label{eq: site percolation bounds on p_{1,c}}
p_{1,c}(\mathbb{Z}^2)\geqslant \left(\theta_{\mathrm{site}}\right)^2+\left(1-\theta_{\mathrm{site}}\right)^2,
\end{align}
where $\theta_{site}$ is the critical value of the $\theta$-parameter for site percolation, i.e. the infimum of $\theta \in [0,1]$ such that switching vertices of $\mathbb{Z}^2$ on independently at random with probability $\theta$ almost surely yields an infinite connected component of on vertices. Plugging in the known rigorous bounds for $0.556\leqslant \theta_{\mathrm{site}}\leqslant 0.679492$~\cite{VDBErmakov96, Wierman95} yields $p_{1,c}(\mathbb{Z}^2)\geqslant 0.5062$, while using the non-rigorous estimate $\theta_{\mathrm{site}}\approx 0.592746$ (see for example \cite{SudingZiff99}) yields the non-rigorous lower-bound $p_{1,c}(\mathbb{Z}^2)\geqslant 0.5172$.

With regards to other lattices, Balister and Bollob\'as completed a rigorous study of $1$-independent percolation models on infinite trees~\cite{BalisterBollobas12}, giving $1$-independent analogues of classical results of Lyons~\cite{Lyons90} for the $0$-independent case. Balister and Bollob\'as's results were later generalised to the $k$-independent setting by Mathieu and Temmel~\cite{MathieuTemmel12}, who also showed interesting links between this problem and theoretical questions concerning the Lov\'asz local lemma, in particular the work of Scott and Sokal~\cite{ScottSokal05, ScottSokal06} on hard-core lattice gases, independence polynomials and the local lemma.

\subsection{Our contributions}
In this paper, we make a three-fold contribution to the study of Problem~\ref{problem: Harris critical probabiliy for 1-ipm}. First of all, we improve previous lower bounds on $p_{1, c}(\mathbb{Z}^2)$ with the following theorems.
\begin{theorem}\label{theorem: local construction integer lattice}
	For all $d \in \mathbb{N}_{\geqslant 2}$, we have that 
	\[p_{1,c}(\mathbb{Z}^{d}) \geqslant  4 - 2 \sqrt{3} \approx 0.535898\ldots \ .\]
\end{theorem}
\noindent Theorem~\ref{theorem: local construction integer lattice} strictly improves on the previous best lower bound for $d=2$ given in (\ref{eq: site percolation bounds on p_{1,c}}) above; moreover, it is based on a very different idea, which first appeared in the second author's PhD thesis~\cite{FalgasRavry12}. In addition 
we give a separate improvement of~(\ref{eq: site percolation bounds on p_{1,c}}): let $\theta_{\mathrm{site}}$ again denote the critical threshold for site percolation. Then the following holds.
\begin{theorem}\label{theorem: global construction integer lattice}
	\[p_{1,c}(\mathbb{Z}^{2}) \geqslant \left(\theta_{\mathrm{site}}(\mathbb{Z}^{2})\right)^{2} + \frac{1}{2}\left(1-\theta_{\mathrm{site}}\left(\mathbb{Z}^{2}\right)\right).\]
\end{theorem}
\noindent Substituting the rigorous bound $\theta_{\mathrm{site}}\geqslant0.556 $ into Theorem~\ref{theorem: global construction integer lattice} yields the lower bound $p_{1,c}(\mathbb{Z}^{2}) \geqslant 0.531136$, which does slightly worse than Theorem~\ref{theorem: local construction integer lattice}. However substituting in the widely believed but non-rigorous estimate $\theta_{\mathrm{site}}\approx 0.592746$ yields a significantly stronger lower bound of $p_{1,c}(\mathbb{Z}^{2}) \geqslant 0.554974$. 

Secondly, motivated by efforts to improve the upper bounds on $p_{1,c}(\mathbb{Z}^2)$, and in particular to establish some $1$-independent analogues of the Russo--Seymour--Welsh (RSW) lemmas on the probability of crossing rectangles, we investigate the following problems. Let $P_n$ denote the graph on the vertex set $\{1,2,\ldots n\}$ with edges $\{12, 23, \ldots, (n-1)n\}$, i.e. a path on $n$ vertices. Given a connected graph $G$, denote by $P_n\times G$ the Cartesian product of $P_n$ with $G$. A \emph{left-right crossing} of $P_n\times G$ is a path from a vertex in $\{1\}\times V(G)$ to a vertex in $\{n\}\times V(G)$. We define the \emph{crossing critical  probability} for $1$-independent percolation on $P_n\times G$ to be
\[p_{1, \times}(P_n\times G):=\inf\Bigl\{p\in[0,1]:\ \forall  \mu \in \mathcal{M}_{1, \geqslant p}(P_n\times G), \ \mu(\exists\textrm{ open left-right crossing})>0 \Bigr\},\]
i.e. the least edge-probability guaranteeing that in any $1$-independent model on $P_n \times G$, there is a strictly positive probability of being able to cross $P_n\times G$ from left to right.
\begin{problem}\label{problem: critical crossing prob}
	Given $n\in \mathbb{N}$ and a finite, connected graph $G$, determine $p_{1, \times}(P_n\times G)$.	
\end{problem}
\noindent Problem~\ref{problem: critical crossing prob} can be thought of as a first step towards the development of $1$-independent analogues of the RSW lemmas; these lemmas play a key role in modern proofs of the Harris--Kesten theorem, and one would expect appropriate $1$-independent analogues to constitute a similarly important ingredient in a solution to Problem~\ref{problem: Harris critical probabiliy for 1-ipm}. By taking the limit as $n\rightarrow \infty$ in Problem~\ref{problem: critical crossing prob}, one is led to consider another $1$-independent critical probability. Let $G$ be an \emph{infinite}, locally finite connected graph. The \emph{long paths critical probability} for $1$-independent percolation on $G$ is
\[p_{1, \ell\mathit{p}}(G):=\inf\Bigl\{p\in[0,1]:\ \forall  \mu \in \mathcal{M}_{1, \geqslant p}(G),\  \forall n\in \mathbb{N}\ \mu(\exists\textrm{ open path of length }n)>0 \Bigr\},\] 
i.e. the least edge-probability at which arbitrarily long open paths will appear in all $1$-independent models in $G$.
\begin{problem}\label{problem: critical long paths prob}
	Given an infinite, locally finite, connected graph $G$, determine $p_{1, \ell\mathit{p}}(G)$.	
\end{problem}
\noindent In this paper, we resolve Problem~\ref{problem: critical crossing prob} in a strong form when $G$ consists of a vertex or an edge (see Theorems \ref{theorem: connected function Pn} and \ref{theorem: ladder construction}). This allows us to solve Problem~\ref{problem: critical long paths prob} when $G$ is the integer line lattice $\mathbb{Z}$ and the integer ladder lattice $\mathbb{Z}\times P_2$.
\begin{theorem}\label{theorem: long paths critical prob}
We have that
	\begin{enumerate}[(i)]
		\item $p_{1, \ell \mathit{p}}(\mathbb{Z})=\frac{3}{4}$, and
		\item $p_{1, \ell \mathit{p}}(\mathbb{Z}\times P_2)=\frac{2}{3}$. 
	\end{enumerate}
\end{theorem}
\noindent Note that part (i) of Theorem~\ref{theorem: long paths critical prob} above can be read out of earlier work of Liggett, Schonman and Stacey~\cite{LiggettSchonmannStacey97} and  Balister and Bollob\'as~\cite{BalisterBollobas12}. We prove further bounds on both $p_{1, \times}(P_n \times G)$ and $p_{1, \ell\mathit{p}}(\mathbb{Z}\times G)$ for a variety of graphs $G$. We summarise the latter, less technical, set of results below. Let $C_{n}$ and $K_n$ denote the cycle and the complete graph on $n$ vertices respectively.
\begin{theorem}\label{theorem: upper bounds on long paths critical prob}
We have that
	\begin{enumerate}[(i)]
		\item $ 0.5359\ldots =4-2\sqrt{3} \leqslant p_{1, \ell\mathit{p}}(\mathbb{Z}\times C_{n})\leqslant p_{1,\ell\mathit{p}}(\mathbb{Z}\times P_n)\leqslant \frac{2}{3} $ for all $n \geqslant 3$;
		\item $p_{1, \ell\mathit{p}}(\mathbb{Z}\times K_3) \leqslant \frac{1}{16} \Bigl(13 - \frac{55}{\sqrt[3]{128 \sqrt{14} - 251}} + \sqrt[3]{128 \sqrt{14} - 251}\Bigr)= 0.63154\ldots $\ ;
		\item $p_{1, \ell\mathit{p}}(\mathbb{Z}\times C_4) \leqslant (3-\sqrt{3})/2=0.63397\ldots$\ ;
		\item $p_{1, \ell\mathit{p}}(\mathbb{Z}\times C_5) \leqslant 0.63895\ldots$\ ;
		\item $0.5359\ldots=4-2\sqrt{3}\leqslant \lim_{n\rightarrow \infty} p_{1, \ell\mathit{p}}(\mathbb{Z}\times K_{n})\leqslant \frac{5}{9}=0.5555\ldots$.
	\end{enumerate}	
\end{theorem}

A key ingredient in the proof of Theorems~\ref{theorem: long paths critical prob} and~\ref{theorem: upper bounds on long paths critical prob} is a local lemma-type result, Theorem~\ref{theorem: crossing G times Pn}, relating the probability in a $1$-independent model of finding an open left-right crossing of $P_n\times G$ to the probability of a given copy of $G$ being connected in that model. This motivated our third contribution to the study of $1$-independent models in this paper, namely an investigation into the connectivity of $1$-independent random graphs.
\begin{definition}
	Let $G$ be a finite connected graph. For any $p\in [0,1]$, we define the $k$-independent connectivity function of $G$ to be
	\[f_{k,G}(p):=\inf\Bigl\{\mu(\exists \textrm{ open spanning tree}): \ \mu \in \mathcal{M}_{k,\geqslant p}(G)\Bigr\}.\]	
\end{definition}
\begin{problem}\label{problem: connectivity}
	Given a finite connected graph $G$, determine $f_{1,G}(p)$.
\end{problem}
\noindent We resolve Problem~\ref{problem: connectivity} exactly when $G$ is a path, a complete graph or a cycle on at most $5$ vertices.
\begin{theorem}\label{theorem: connected function Pn}
	Given $n \in \mathbb{N}_{\geqslant 2}$ and $p \in [0,1]$, let $\theta = \theta(p) :=\frac{1 + \sqrt{4p - 3}}{2}$ and $p_{n}:= \frac{1}{4}\left(3 - \tan^{2}\left(\frac{\pi}{n+1} \right) \right)$.  We have that
		\begin{equation}
	f_{1,P_{n}}(p) = \begin{cases} \sum_{j = 0}^{n} \theta^{j}(1-\theta)^{n-j} &\text{ for } p \in [p_{n},1] , \\
	0 &\text{ for } p \in [0,p_{n}]. \end{cases} \nonumber
	\end{equation}
\end{theorem}
\begin{theorem}\label{theorem: connected function Kn}
	Given $n \in \mathbb{N}_{\geqslant 2}$ and $p \in [0,1]$, let $\theta = \theta(p) := \frac{1 + \sqrt{2p - 1}}{2}$ and $p_{n} := \frac{1}{2}(1 - \tan^{2}(\frac{\pi}{2n}))$.  We have that
	\begin{equation}
	f_{1,K_{n}}(p) = \begin{cases} \theta^{n} + (1-\theta)^{n} &\text{ for } p \in [p_{n},1] , \\
	0 &\text{ for } p \in [0,p_{n}]. \end{cases} \nonumber
	\end{equation}
	In particular,
	\begin{equation}
	f_{1, K_3}(p)=f_{1, C_3}(p)= \begin{cases}	
	\frac{3p-1}{2}  & \text{for } p \in [\frac{1}{3},1], \\
	0  & \text{for } p \in [0,\frac{1}{3}]. 
	\end{cases} \nonumber
	\end{equation}	
\end{theorem}
\begin{theorem}\label{theorem: connected function C4}
	For $p \in [0,1]$ we have that 
	\begin{equation}
	f_{1, C_4}(p)= \begin{cases}	
	2p-1  & \text{for } p \in [\frac{1}{2},1], \\
	0  & \text{for } p \in [0,\frac{1}{2}]. 
	\end{cases} \nonumber
	\end{equation}
	\end{theorem}
\begin{theorem}\label{theorem: connected function C5}
	For $p \in [0,1]$ we have that
	\begin{equation}
	f_{1, C_5}(p)= \begin{cases} 
	\frac{p(3p^2-1)}{3p-1}  & \text{for } p \in [\frac{\sqrt{3}}{3},1], \\
	0  & \text{for } p \in [0,\frac{\sqrt{3}}{3}]. 
	\end{cases} \nonumber
	\end{equation}
\end{theorem}
\noindent We also consider the opposite problem to Problem~\ref{problem: connectivity}, namely maximising connectivity in $1$-independent random graph models. Let $\mathcal{M}_{k, \leqslant p}(G)$ denote the collection of $1$-independent measures $\mu$ on $G$ such that $\sup_{e\in E(G)}\mu \{e \textrm{ is open}\}\leqslant p$. Set
\[F_{k, G}(p):=\sup \{  \mu \left(\exists \textrm{ open spanning tree}\right): \ \mu \in \mathcal{M}_{k,\leqslant p}(G) \}.\]
\begin{problem}\label{problem: maximising connectivity}
	Given a finite connected graph $G$, determine $F_{1,G}(p)$.
\end{problem}
\noindent We resolve Problem~\ref{problem: maximising connectivity} exactly when $G$ is a path, a complete graph or a cycle on at most $5$ vertices.
\begin{theorem}\label{theorem: maximising connectivity in Pn}
	For all $n\in \mathbb{N}$ with $n\geqslant 2$, $F_{1, P_n}(p)=p^{\lfloor \frac{n}{2}\rfloor}$.
\end{theorem}
\begin{theorem}\label{theorem: maximising connectivity in Kn}
	For all $n\in \mathbb{N}$ with $n\geqslant 2$, $F_{1, K_n}(p)=1-f_{1, K_n}(1-p)$.
\end{theorem}
\begin{theorem}\label{theorem: maximising connectivity for C4}
\[F_{1, C_4}(p)=\left\{ \begin{array}{ll}
2p-p^2 & \textrm{if }p\in [\frac{2}{3}, 1],\\
2p^2 & \textrm{if }p\in [0, \frac{1}{3}].
\end{array}\right.	\]	
\end{theorem}
\begin{theorem}\label{theorem: maximising connectivity for C5}
	\[F_{1, C_5}(p)=\left\{ \begin{array}{ll}
	\frac{p(2-5p(1-p))}{5-3p} & \textrm{if }p\in [\frac{3}{5}, 1],\\
	\frac{5p^2}{3}&\textrm{if }p\in [\frac{1}{2}, \frac{3}{5}],\\
	\frac{5p^2(p+1)}{p+4}& \textrm{if }p\in [0, \frac{1}{2}].
	\end{array}\right.	\]	
\end{theorem}
\noindent Together, Theorems~\ref{theorem: connected function Pn}--\ref{theorem: connected function C5} and~\ref{theorem: maximising connectivity in Kn}--\ref{theorem: maximising connectivity for C5}  determine the complete connectivity `profile' for $1$-independent measures $\mu$ on $K_n$, $P_n$, $C_4$ and $C_5$ --- that is, the range of values $\mu(\{\textrm{connected}\})$ can take if every edge is open with probability $p$. In Figure~\ref{figure: connectivity}, we illustrate these for four of these graphs $G$ with plots of $f_{1,G}(p)$, $F_{1, G}(p)$ and $f_{0, G}(p):=\mu\left(\mathbf{G}_p \textrm{ is connected}\right)$, where $\mathbf{G}_p$ is the $0$-independent model on $G$ obtained by setting each edge of $G$ to be open with probability exactly $p$, independently at random.

\begin{figure}
\begin{tikzpicture}
\begin{axis}[
title={Possible connectivity for $K_3$},
axis lines = left,
xlabel = $p$,
ylabel = {$f(p)$},
legend pos=north west,
width=0.45\textwidth,
height=0.45\textwidth,
]
\addplot [
domain=0:2/3, 
samples=100, 
color=red,
]
{3*x/2};
\addplot [
domain=2/3:1, 
samples=100, 
color=blue,
]
{1};
\addplot [
domain=1/3:1, 
samples=100, 
color=green,
]
{(3*x-1)/2};
\addplot [
domain=0:1/3, 
samples=100, 
color=green,
]
{0};
\addplot [
domain=0:1, 
samples=100, 
style=dashed,
]
{x^3+3*x^2*(1-x)};
\end{axis} 
\end{tikzpicture}
\begin{tikzpicture}
\begin{axis}[
title={Possible connectivity for $K_4$},
axis lines = left,
xlabel = $p$,
ylabel = {$f(p)$},
legend pos=north west,
width=0.45\textwidth,
height=0.45\textwidth,
]
\addplot [
domain=0:1-0.41421, 
samples=100, 
color=red,
]
{(4*x-x^2)/2};
\addplot [
domain=1-0.41421:1, 
samples=100, 
color=blue,
]
{1};
\addplot [
domain=0.41421:1, 
samples=100, 
color=green,
]
{(x^2+2*x-1)/2};
\addplot [
domain=0:1, 
samples=100, 
style=dashed,
]
{x^6+6*x^5*(1-x)+15*x^4*(1-x)^2+16*x^3*(1-x)^3};
\addplot [
domain=0:0.41421, 
samples=100, 
color=green,
]
{0};
\end{axis} 
\end{tikzpicture}
\begin{tikzpicture}
\begin{axis}[
title={Possible connectivity for $C_4$},
axis lines = left,
xlabel = $p$,
ylabel = {$f(p)$},
legend pos=north west,
width=0.45\textwidth,
height=0.45\textwidth,
]
\addplot [
domain=0:2/3, 
samples=100, 
color=red,
]
{2*x^2};
\addplot [
domain=2/3:1, 
samples=100, 
color=blue,
]
{2*x-x^2};
\addplot [
domain=1/2:1, 
samples=100, 
color=green,
]
{2*x-1};
\addplot [
domain=0:1, 
samples=100, 
style=dashed,
]
{x^4+4*x^3*(1-x)};
\addplot [
domain=0:1/2, 
samples=100, 
color=green,
]
{0};
\end{axis} 
\end{tikzpicture}
\begin{tikzpicture}
\begin{axis}[
title={Possible connectivity for $C_5$},
axis lines = left,
xlabel = $p$,
ylabel = {$f(p)$},
legend pos=north west,
width=0.45\textwidth,
height=0.45\textwidth,
]
\addplot [
domain=0:1/2, 
samples=100, 
color=red,
]
{5*x^2*(x+1)/(x+4)};
\addplot [
domain=1/2:3/5, 
samples=100, 
color=blue,
]
{5*x^2/3};
\addplot [
domain=3/5:1, 
samples=100, 
color=purple,
]
{x*(5*x^2-5*x-2)/(3*x-5)};
\addplot [
domain=0.57735:1, 
samples=100, 
color=green,
]
{x*(3*x^2-1)/(3*x-1)};
\addplot [
domain=0:1, 
samples=100, 
style=dashed,
]
{x^5+5*x^4*(1-x)};
\addplot [
domain=0:0.57735, 
samples=100, 
color=green,
]
{0};
\end{axis} 
\end{tikzpicture}
\caption{The $1$-independent connectivity profile of $G$ for $G=K_3$ $K_4$, $C_4$ and $C_5$. The green curve represents $f_{1,G}(p)$, the dashed black curve $f_{0, G}(p)$, and the union of the red, blue and purple segments represent the piecewise smooth function $F_{1,G}(p)$.}\label{figure: connectivity}
\end{figure}
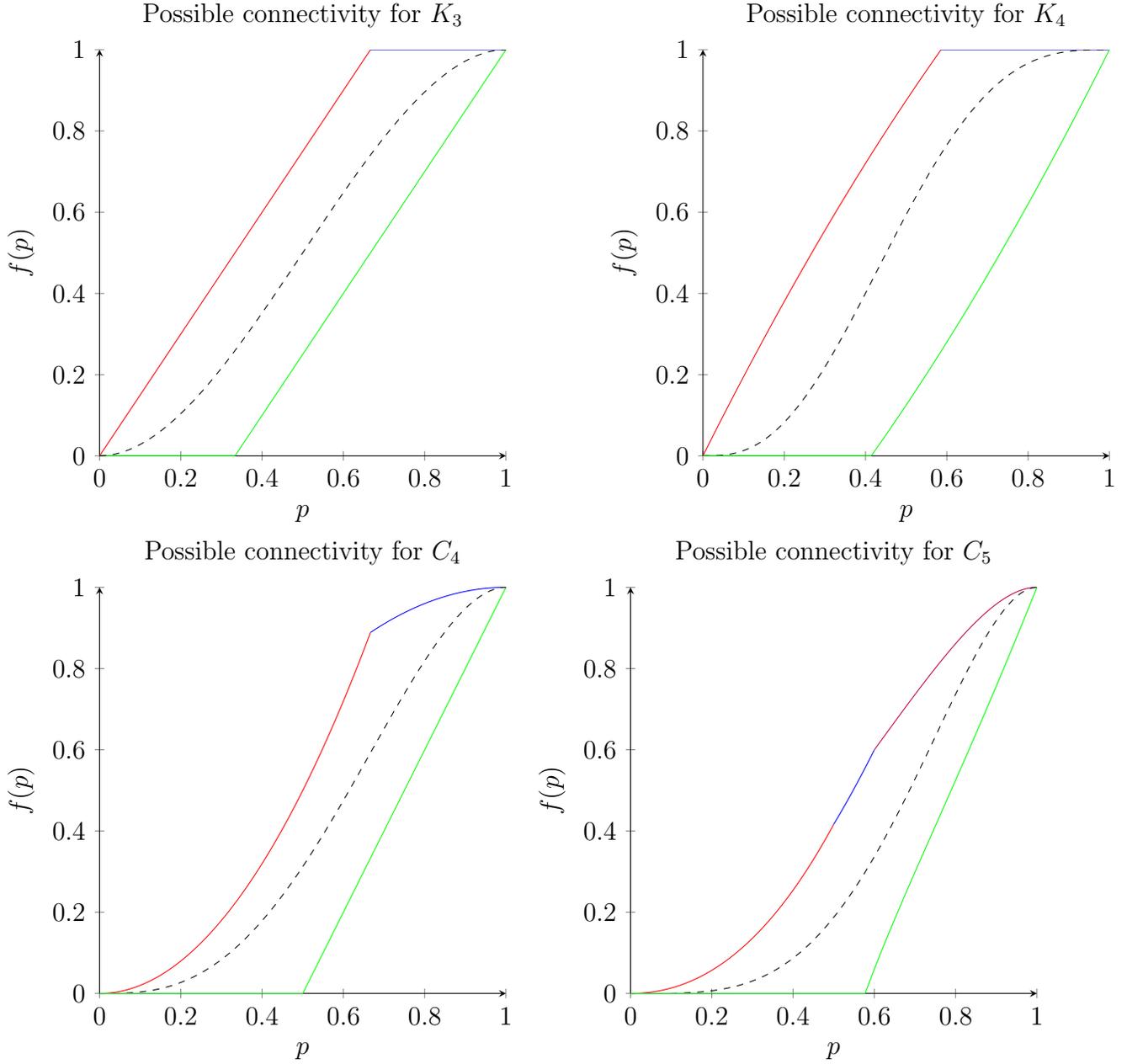

\subsection{Organisation of the paper}
Our first set of results, Theorems~\ref{theorem: local construction integer lattice} and~\ref{theorem: global construction integer lattice} are proved in Section~\ref{section: lower bound on crit prob for percolation in Z^2}.

In Section~\ref{section: general upper bound on crossing prob}, we use arguments reminiscent of those used in inductive proofs of the Lov\'asz local lemma to obtain Theorem~\ref{theorem: crossing G times Pn}, which gives a general upper bound for crossing and long paths critical probabilities in $1$-independent percolation models on Cartesian products  $\mathbb{Z}\times G$. This result is used in Sections~\ref{section: line lattice} and~\ref{section: ladder} to prove Theorem~\ref{theorem: long paths critical prob} on the long paths critical probability for the line and ladder lattices.

In Sections~\ref{section: line lattice}, \ref{section: complete graphs} and~\ref{section: cycles} and~\ref{section: maximising connectivity}, we prove our results on $f_{1,G}(p)$ and $F_{1,G}(p)$
when $G$ is a path, a complete graph or a short cycle. 
We apply these results in Section~\ref{section: long paths} to prove Theorem~\ref{theorem: upper bounds on long paths critical prob}. Finally we end the paper in Section~\ref{section: open problems} with a discussion of the many open problems arising from our work.

\subsection{Notation}\label{subsection: notation}

We write $\mathbb{N}$ for the set of natural numbers $\{1,2,\ldots\}$, $\mathbb{N}_{0}$ for the set $\mathbb{N} \cup \{0\}$, and $\mathbb{N}_{\geqslant k}$ for the set of natural numbers greater than or equal to $k$.

We set $[n]:=\{1,2, \ldots n\}$. Given a set $A$, we write $A^{(r)}$ for the collection of all subsets of $A$ of size $r$, hereafter referred to as $r$-sets from $A$. We use standard graph theoretic notation. A graph is a pair $G=(V,E)$ where $V=V(G)$ and $E=E(G)\subseteq V(G)^{(2)}$ denote the vertex set and edge set of $G$ respectively. Given a subset $A\subseteq G$, we denote by $G[A]$ the subgraph of $G$ induced by $A$.  We also write $N(A)$ for the set of vertices in $G$ adjacent to at least one vertex in $A$.

Given two graphs $G$ and $H$, we write $G \times H$ for the Cartesian product of $G$ with $H$, which is the graph on the vertex set $V(G) \times V(H)$ having an edge between $(x,u)$ and $(y,v)$ if and only if either $u=v$ and $x$ is adjacent to $y$ in $G$, or $x=y$ and $u$ is adjacent to $v$ in $H$.

Throughout this paper, we shall use \emph{$k$-ipm} as a shorthand for `$k$-independent percolation model/measure'. In a slight abuse of language, we say that a bond percolation model $\mu$ on an infinite connected graph $G$ percolates if $\mu(\{\textrm{percolation}\})=1$. We refer to a random configuration $\mathbf{G}_{\mu}$ as a $\mu$-random subgraph of $G$. Finally we write $\mathbb{E}_{\mu}$ for the expectation taken with respect to the probability measure $\mu$.  For any event $X$, we write $X^{c}$ for the complement event.

\section{Lower bounds on $p_{1,c}(\mathbb{Z}^d)$}\label{section: lower bound on crit prob for percolation in Z^2}

\begin{proof}[Proof of Theorem \ref{theorem: local construction integer lattice}]
Let $d\in \mathbb{N}_{\geqslant 2}$. For $k \in \mathbb{N}_{0}$, let $T_{k} := \big\{(x,y) \in \mathbb{Z}^{d}: \max(|x|,|y|) = k\big\}$.  Let $q := \sqrt{3} - 1$.  For each vertex in $\mathbb{Z}^{d}$, we colour it either Blue or Red, or set it to state $I$, which stands for \textit{Inwards}.  The probability that a given vertex will be in each of these states will depend on which of the $T_{k}$ the vertex is in, and we assign these states to each vertex independently of all other vertices.
\begin{itemize}
	\item If $v$ is a vertex in $T_{k}$, where $k \equiv 0 \!\mod 6$, then we colour $v$ Blue. 
	\item If $v$ is a vertex in $T_{k}$, where $k \equiv 1 \!\mod 6$, then we colour $v$ Red with probability $q/2$ and colour it Blue otherwise.
	\item If $v$ is a vertex in $T_{k}$, where $k \equiv 2 \!\mod 6$, then we colour $v$ Red with probability $q$ and put it in the Inwards state $I$ otherwise.
	\item If $v$ is a vertex in $T_{k}$, where $k \equiv 3 \!\mod 6$, then we colour $v$ Red. 
	\item If $v$ is a vertex in $T_{k}$, where $k \equiv 4 \!\mod 6$, then we colour $v$ Blue with probability $q/2$ and colour it Red otherwise.
	\item If $v$ is a vertex in $T_{k}$, where $k \equiv 5 \!\mod 6$, then we colour $v$ Blue with probability $q$ and put it in the Inwards state $I$ otherwise.
\end{itemize}
Note that the rules for $T_{k+3}, T_{k+4}, T_{k+5}$ are the same as those for $T_{k}, T_{k+1}, T_{k+2}$ respectively, except with red and blue interchanged. See Figure \ref{figure: vertex states} for the possible states of the vertices in $T_{0},T_{1},T_{2}$ and $T_{3}$ when $d=2$. 
\begin{figure}
\begin{center}
$\begin{array}{ccccccc}
{\color{red}R} & {\color{red}R}  & {\color{red}R}  & {\color{red}R}  & {\color{red}R}  & {\color{red}R}  & {\color{red}R}\\
{\color{red}R} & {\color{red}R} I & {\color{red}R} I & {\color{red}R} I & {\color{red}R} I & {\color{red}R} I & {\color{red}R}\\
{\color{red}R} & {\color{red}R} I & {\color{blue}B} {\color{red}R} & {\color{blue}B} {\color{red}R} & {\color{blue}B} {\color{red}R} & {\color{red}R} I & {\color{red}R}\\
{\color{red}R} & {\color{red}R} I &  {\color{blue}B} {\color{red}R} & {\color{blue}B} & {\color{blue}B} {\color{red}R} & {\color{red}R} I & {\color{red}R}\\
{\color{red}R} & {\color{red}R} I & {\color{blue}B} {\color{red}R} & {\color{blue}B} {\color{red}R} & {\color{blue}B} {\color{red}R} & {\color{red}R} I & {\color{red}R}\\
{\color{red}R} & {\color{red}R} I & {\color{red}R} I & {\color{red}R} I & {\color{red}R} I & {\color{red}R} I & {\color{red}R}\\
{\color{red}R} & {\color{red}R}  & {\color{red}R}  & {\color{red}R}  & {\color{red}R}  & {\color{red}R}  & {\color{red}R}\\
\end{array}$
\caption{The possible states of the vertices in $T_{0},T_{1},T_{2}$ and $T_{3}$ when $d=2$.  The letter $B$ stands for Blue, the letter $R$ stands for Red, and the letter $I$ stands for the Inwards state.}\label{figure: vertex states}
\end{center}
\end{figure}
Now suppose that $e = \{v_{1},v_{2}\}$ is an edge in $\mathbb{Z}^{d}$.  Firstly we say that the edge $e$ is open if either both $v_{1}$ and $v_{2}$ are Blue or both $v_{1}$ and $v_{2}$ are Red.  We also say the edge $e$ is open if, for some $k$, we have that $v_{1} \in T_{k}$, $v_{2} \in T_{k+1}$, and $v_{2}$ is in state $I$.  In all other cases we say that the edge $e$ is closed.  It is clear that this gives a $1$-independent measure on $\mathbb{Z}^{d}$ as it is vertex-based, and it is also easy to check that every edge is present with probability at least $4 - 2 \sqrt{3}$.

Call this measure $\mu$, and let $G:=\mathbb{Z}^d$.
We claim that in $\mathbf{G}_{\mu}$, for all $k \equiv 0 \!\mod 3$, there is no path of open edges from $T_{k}$ to $T_{k+3}$.  Suppose this is not the case, and $P$ is some path of open edges from a vertex in $T_{k}$ to $T_{k+3}$.  We first note that $P$ cannot include a vertex in state $I$, as such a vertex would be in $T_{k+2}$ and would only be adjacent to a single edge.  Thus every vertex of $P$ is either Blue or Red.  However, as one end vertex of $P$ is Blue and the other end vertex is Red, and there are no open edges with different coloured end vertices, we have that such a path $P$ cannot exist.  As a result, every component of $\mathbf{G}_{\mu}$ is sandwiched between some $T_{k-3}$ and $T_{k+3}$, where $k \equiv 0 \!\mod 3$, and so is of finite size.  Thus  we have that $p_{1,c}(\mathbb{Z}^{d}) \geqslant  4 - 2 \sqrt{3}$.
\end{proof}

The construction in Theorem \ref{theorem: local construction integer lattice} can in fact be generalised to certain other graphs and lattices.  Given an infinite, connected, locally finite graph $G$, and a vertex set $A \subseteq V(G)$, let $\overline{A}$ be the closure of $A$ under $2$-neighbour bootstrap percolation on $G$.  That is, let $\overline{A} := \bigcup_{i \geqslant 0} A_{i}$, where $A_{0} := A$ and for $i \geqslant 1$
\begin{equation}
A_{i} := A_{i-1} \cup \{v \in V(G): v \text{ has } 2 \text{ or more neighbours in } A_{i-1}\}. \nonumber
\end{equation}
We say that $G$ has the \textit{finite $2$-percolation property} if, for every finite set $A \subseteq V(G)$, we have that $\overline{A}$ is finite.  
\begin{corollary}\label{corollary: if 2-percolation property, then lb on p_{1,c}}
If $G$ has the finite $2$-percolation property, then $p_{1,c}(G) \geqslant 4 - 2\sqrt{3}$.	
\end{corollary}
\begin{proof}
Partition $V(G)$ in the following way: pick any vertex $v$ and set $T_{0} := \{v\}$.  For $k \geqslant 1$ let
\begin{equation}
T_{k} := \overline{N(T_{k-1})} \setminus \bigcup_{j = 0}^{k-1} T_{j}\nonumber.
\end{equation}
We have that if $w \in T_{k}$, then $w$ is only adjacent to vertices in $T_{k-1},T_{k}$ and $T_{k+1}$.  Moreover, $w$ is adjacent to at most one vertex in $T_{k-1}$ --- this is the crucial property needed for our construction.  Since $G$ has the finite $2$-percolation property, each $T_{k}$ is finite.  Thus we can use the $T_{k}$ to construct a non-percolating $1$-ipm on $G$ in the exact same fashion as done for $\mathbb{Z}^{d}$ in Theorem \ref{theorem: local construction integer lattice} (the key being that vertices in state $I$ are still dead ends, being incident to a unique edge), which in turn shows that $p_{1,c}(G) \geqslant 4 - 2\sqrt{3}$.  
\end{proof}
An example of a lattice with the finite $2$-percolation property is the lattice $(3,4,6,4)$, where here we are using the lattice notation of Gr\"{u}nbaum and Shephard \cite{GrunbaumShephard86}.  Riordan and Walters \cite{RiordanWalters07} showed that the site percolation threshold of this lattice is very likely to lie in the interval $[0.6216,0.6221]$.  Thus this estimate, together with Newman's construction (see equation (\ref{eq: site percolation bounds on p_{1,c}})), shows (non-rigorously) that $p_{1,c}\left(\left(3,4,6,4\right)\right) \geqslant 0.52981682$.  As this is less than $4-2\sqrt{3}$, we have that our construction gives the (rigorous) improvement of $p_{1,c}\left(\left(3,4,6,4\right)\right) \geqslant 4-2\sqrt{3}$.

\begin{proof}[Proof of Theorem \ref{theorem: global construction integer lattice}]

Fix $\varepsilon >  0$ sufficiently small so that $q := \theta_{\text{site}}(\mathbb{Z}^{2}) - \varepsilon$ is strictly larger than $1/4$.  For each vertex $v \in \mathbb{Z}^{2}$, we assign to it one of three states: On, $L$ or $D$, and we do this independently for every vertex.  We assign $v$ to the On state with probability $q$, we assign it to the $L$ state with probability $\frac{1}{2}(1-q)$, and else we assign it to the $D$ state with probability $\frac{1}{2}(1-q)$.

We now describe which edges are open based on the states of the vertices.  We first say that the edge $e$ is open if both of its vertices are in the On state.  If a vertex is in state $L$, then the edge adjacent and to the left of it is open.  Similarly, if a vertex is in state $D$, then the edge adjacent and down from it is open.  All other edges are closed.  See Figure \ref{figure: global construction} for an example of this construction.

\begin{figure}[ht]
	\centering
\includegraphics[scale=1]{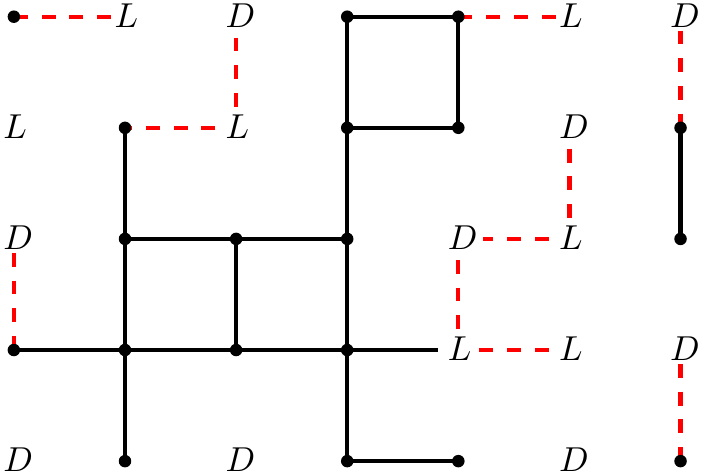}
	\caption{This figure shows the open edges of the construction on a small subset of $\mathbb{Z}^2$.  The unlabelled vertices correspond to those in the On state.  The black edges are the open edges that are adjacent to two On vertices, while the dashed red edges are the open edges that are either to the left of an $L$ vertex or below a $D$ vertex.}
	\label{figure: global construction}
\end{figure}	

It is easy to see that this is a $1$-independent measure on $\mathbb{Z}^{2}$ as it is vertex-based, and every edge is present with probability $q^{2} + \frac{1}{2}(1-q)$. Call this measure $\mu$ and let $G:=\mathbb{Z}^2$. We will show that every component of $\mathbf{G}_\mu$ has finite size.  We begin by first proving an auxiliary lemma.  Let $t \in [0,\frac{1}{2}]$, and let us define another $1$-independent measure on $\mathbb{Z}^{2}$, which we call the \textit{left-down} measure with parameter $t$.  In the left-down measure, each vertex of $\mathbb{Z}^{2}$ is assigned to one of three states: Off, $L$ or $D$, and we do this independently for every vertex.  For each vertex $v \in \mathbb{Z}^{2}$, we assign it to state $L$ with probability $t$, we assign it to state $D$ with probability $t$, and we assign it to state Off with probability $1-2t$.  As above, if a vertex is in state $L$, then the edge adjacent and to the left of it is open, while if a vertex is in state $D$, then the edge adjacent and down from it is open.  All other edges are closed. We use $\nu_t$ to denote the left-down measure with parameter $t$.

\begin{lemma}\label{lemma: left-down measure}
If $ 0 \leqslant t \leqslant \frac{3}{8}$, then all components in $\mathbf{G}_{\nu_t}$ are finite almost surely.
\end{lemma}

\begin{proof}
Let $z := 1 - \sqrt{1-2t}$.  As $0 \leqslant t \leqslant \frac{3}{8}$ we have that $0 \leqslant z \leqslant \frac{1}{2}$.  We start by taking a random subgraph of $\mathbb{Z}^{2}$ where every edge is open with probability $z$, independently of all other edges.  We then further modify it as follows.  For each vertex $v = (x,y)$ we look at the state of the edge $e_{1}$ from $v$ to the vertex $(x-1,y)$, and the state of the edge $e_{2}$ from $v$ to the vertex $(x,y-1)$.  If at least one of $e_{1}$ or $e_{2}$ is closed, we do not change anything.  However, if both $e_{1}$ and $e_{2}$ are open, with probability $\frac{1}{2}$ we close the edge $e_{1}$, and otherwise we close the edge $e_{2}$.  We do this independently for every vertex $v$ of $\mathbb{Z}^{2}$.

It is easy to see that this is an equivalent formulation of $\nu_t$, the left-down measure with parameter $t$. Indeed, to each vertex $v=(x,y)$ as above we may assign a state Off if both the edge $e_1(v)$ to the vertex to the left of $v$ and the edge $e_2(v)$ to the vertex below $v$ are closed, a  state $L$ if $e_1(v)$ is open and  a state $D$ if $e_2(v)$ is open.  The probabilities of these three states are $(1-z)^2=1-2t$, $t$ and $t$ respectively, and since the vertex states depend only on the pairwise disjoint edge sets $\{e_1(v), e_2(v)\}_{v\in \mathbb{Z}^2}$, they are independent of one another just as in the $\nu_t$ measure.

  Thus we have coupled $\nu_t$ to the $0$-independent bond percolation measure $\xi$ on $\mathbb{Z}^{2}$ with edge-probability $z$.  In this coupling we have that if an edge $e$ is open in $\mathbf{G}_{\nu_t}$, then it is also open in $\mathbf{G}_{\xi}$.  As $z \leqslant 0.5$ we have that all components in $\mathbf{G}_{\xi}$ are finite by the Harris--Kesten theorem, and so we also have that all of the components in $\mathbf{G}_{\nu_t}$ are finite too.
\end{proof}

By considering an appropriate branching process it is possible to prove the stronger result that if $ 0 \leqslant t < \frac{1}{2}$, then almost surely all components in $\mathbf{G}_{\nu_t}$ are finite.  We make no use of this stronger result in this paper, so we omit its proof.  It is also clear that when $t = \frac{1}{2}$, every vertex in $\mathbf{G}_{\nu_t}$  is part of an infinite path consisting solely of steps to the left or steps downwards, and so percolation occurs in $\mathbf{G}_{\nu_t}$ at this point.

Let us return to our original $1$-independent measure $\mu$, where every vertex is in state On, $L$ or $D$.  Recall that our aim is to show that all components have finite size in $\mathbf{G}_\mu$.  Consider removing all vertices in state $L$ or $D$, and also any edges adjacent to these vertices.  What is left will be a collection of components consisting only of edges between vertices in the On state, which we call the \textit{On-sections}.  The black edges in Figure \ref{figure: global construction} are the edges in the On-sections.  As a vertex is On with probability $q < \theta_{\text{site}}(\mathbb{Z}^{2})$, we have that almost surely every On-section is finite.  Similarly, consider removing all edges in the On-sections.  What is left will be a collection of edges adjacent to vertices in the $L$ or $D$ states.  We call these components the $LD$-sections; the dashed red edges in Figure \ref{figure: global construction} are the edges in the $LD$-sections.  As each vertex is in state $L$ with probability $\frac{1}{2}(1-q) \leqslant \frac{3}{8}$ and in state $D$ with the same probability, Lemma \ref{lemma: left-down measure} tells us that almost surely every $LD$-section is finite.

For each vertex $v$ in state $L$ orient the open edge to the left of it away from $v$, while for each vertex $v$ in state $D$ orient the open edge below it away from $v$.  This gives a partial orientation of the open edges of $\mathbf{G}_\mu$, in which every vertex in state $L$ or $D$ has exactly one edge oriented away from it, and vertices in state On have no outgoing edge. Furthermore, if $v_{1}$ is a vertex in the On state and $v_{2}$ is a vertex in the $L$ or $D$ state, then the edge between them is oriented from $v_{2}$ to $v_{1}$.  Since the $LD$-sections are almost surely finite, this implies the $LD$ sections under this orientation consist of directed trees, each of which is oriented from the leaves to a unique root, which is in the On state. In particular, every $LD$-section attaches to at most one On-section.  As such, almost surely every component in $\mathbf{G}_\mu$ consists of at most one On-section, and a finite number of finite $LD$-sections attached to it.  Thus almost surely every component in $\mathbf{G}_\mu$ is finite.
\end{proof}


\section{A general upper bound for $p_{1, \ell p}\left(\mathbb{Z}\times G\right)$}\label{section: general upper bound on crossing prob}

Let $G$ be a finite connected graph. Set $v(G):=\vert V(G)\vert$. Recall that for any $1$-independent bond percolation measure $\mu \in \mathcal{M}_{1, \geqslant p}(G)$, we have $\mu(\mathbf{G}_{\mu} \textrm{ is connected})\geqslant f_{1,G}(p)$. 

\begin{theorem}\label{theorem: crossing G times Pn}
	If $p$ satisfies 
	\begin{align}\label{equation: bounding equation} 
	\left(f_{1,G}(p)\right)^2\geqslant \frac{1}{\alpha(1-\alpha)} (1-p)^{v(G)}, 
	\end{align}
	for some $\alpha\in(0,1/2]$, then for every $\ell\in \mathbb{N}$
	\[f_{1, P_{\ell}\times G}(p)\geqslant \left((1-\alpha)f_{1,G}(p)\right)^{\ell} .\]	
\end{theorem}
\begin{proof}
	Consider an arbitrary measure $\mu\in \mathcal{M}_{1,\geqslant p}(\mathbb{Z}\times G)$. For any $n\in \mathbb{N}$, the restriction of $\mu$ to $[n]\times V(G)$ is a measure from $\mathcal{M}_{1,\geqslant p}(P_n\times G)$, and clearly all such measures can be obtained in this way. Furthermore, for every $n\in \mathbb{N}$, the restriction of $\mu$ to $\{n\}\times V(G)$ is a measure from $\mathcal{M}_{1,\geqslant p}(G)$, and in particular the subgraph of $\mathbf{(\mathbb{Z}\times G)}_{\mu}$ induced by $\{n\}\times V(G)$ is connected with probability at least $f_{1, G}(p)$.

	We consider the $\mu$-random graph $\mathbf{(\mathbb{Z}\times G)}_{\mu}$. For $n\geqslant 1$ let $Y_n$ be the event that $[n]\times V(G)$ induces a connected subgraph. For $n\geqslant 2$, let $X_n$ be the event that $[n-1]\times V(G)$ induces a connected subgraph and at least one vertex in $\{n\}\times V(G)$ is connected to a vertex in $\{n-1\}\times V(G)$. For $n=1$, set $X_1$ to be the trivially satisfied event occurring with probability $1$. 
	For $n\geqslant 1$, let $V_n$ be the event that $\{n\}\times V(G)$ induces a connected subgraph, and for $n\geqslant 2$ let $H_n$ be the event that at least one of the edges from $\{n-1\}\times V(G)$ to $\{n\}\times V(G)$ is present. 
	
	It easily follows that $X_n = Y_{n-1} \cap H_n$ and $X_n \cap V_n \subseteq Y_n$. From here, 
we obtain the following inclusions:\begin{enumerate} [(a)]
		\item $(X_{n+1})^c\cap Y_n = (H_{n+1})^c\cap Y_n$,   
		\item $Y_n\cap Y_{n-1}\supseteq (V_n\cap X_n)\cap Y_{n-1}$,  and
		\item  $(Y_n)^c\cap X_{n}\subseteq (V_n)^c \cap X_n$.  
	\end{enumerate}
	
	Now set
	\begin{align*}
	x_n &:= \mu\left(X_n^c \vert \bigcap_{m<n} Y_m\right) & \textrm{ and } && y_n &:=\mu\left(Y_n^c\vert X_n \cap (\bigcap_{m<n} Y_m)  \right).
	\end{align*}
	We begin by establishing two inductive relations for the sequences $x_n$ and $y_n$.  First of all, using (a) and (b) we have,
	\begin{align}
	x_{n+1} = \frac{\mu((X_{n+1})^c\cap (\bigcap_{m\leqslant n} Y_m ))} {\mu(\bigcap_{m\leqslant n} Y_m )}
	&=  \frac{\mu((H_{n+1})^c \cap (\bigcap_{m\leqslant n} Y_m ))} {\mu(\bigcap_{m\leqslant n} Y_m )}\notag \\
	&\leqslant  \frac{\mu((H_{n+1})^c ) } {\mu(Y_n \vert (\bigcap_{m<n} Y_m ))} \qquad \textrm{ by $1$-independence}\notag \\
	&\leqslant  \frac{(1-p)^{v(G)}} {\mu(V_n \vert (\bigcap_{m<n} Y_m )) -\mu((X_n)^c \vert (\bigcap_{m<n} Y_m ))}\notag \\
	&\leqslant \frac{(1-p)^{v(G) }}{f_{1,G}(p) -x_n} \qquad \textrm{by $1$-independence}. \label{equation: xn recursive bound}
	\end{align}
	Secondly, using (c),
	\begin{align}
	y_{n} = \frac{\mu((Y_n)^c\cap X_n \cap (\bigcap_{m<n} Y_m ))} {\mu(X_n \cap (\bigcap_{m<n} Y_m ))} &\leqslant  \frac{\mu((V_n)^c \cap (\bigcap_{m<n} Y_m ))} {\mu(X_n \cap (\bigcap_{m<n} Y_m ))}\notag \\
	&\leqslant  \frac{\mu((V_n)^c ) } {\mu(X_n \vert (\bigcap_{m<n} Y_m ))} \qquad \textrm{ by $1$-independence}\notag \\
	&	\leqslant  \frac{(1-f_{1,G}(p))} {1-x_n}. \label{equation: yn recursive bound}
	\end{align}
	Now if (\ref{equation: bounding equation}) is satisfied, we claim that $x_n\leqslant \alpha f_{1,G}(p)$ for all $n$. Indeed $x_1=0$, and if $x_n \leqslant \alpha f_{1,G}(p)$, then by (\ref{equation: yn recursive bound}) 
	\begin{align} y_n\leqslant\frac{1-f_{1,G}(p)}{1-\alpha f_{1,G}(p)}= 1- \frac{(1-\alpha)f_{1,G}(p)}{1-\alpha f_{1,G}(p)}<1. \nonumber \end{align}
	Furthermore, we have by (\ref{equation: xn recursive bound}) and (\ref{equation: bounding equation}) that
	\begin{align} x_{n+1}\leqslant \frac{(1-p)^{v(G)}}{(1-\alpha)f_{1,G}(p)}\leqslant \alpha f_{1,G}(p), \nonumber\end{align}
	so our claim follows by induction.

	Finally, we have that
	\begin{align*}
	\mu(Y_{\ell})= \prod_{i=1}^{\ell} (1-x_i)(1-y_i)&> \left(\frac{(1-\alpha)f_{1,G}(p)}{1-\alpha f_{1,G}(p)} \right)^{\ell}\left(1-\alpha f_{1,G}(p)\right)^{\ell}= \left( (1-\alpha) f_{1,G}(p)\right)^{\ell}.
	\end{align*}
	
\end{proof}
For any finite connected graph $G$, $f_{1,G}(p)$ is a non-decreasing function of $p$ with $f_{1,G}(p)=1$. Thus the function $(f_{1,G}(p))^2$ is also non-decreasing in $p$ and attains a maximum value of $1$ at $p=1$. On the other hand, the function $4(1-p)^{v(G)}$ is strictly decreasing in $p$ and is equal to $4$ at $p=0$. Thus there exists a unique solution $p_{\star}=p_{\star}(G)$ in the interval $[0,1]$ to the equation 
\begin{align}\label{equation:  pstar}
(f_{1,G}(p))^2 = 4(1-p)^{v(G)}.
\end{align}
Theorem~\ref{theorem: crossing G times Pn} thus has the following immediate corollary.
\begin{corollary}\label{corollary: bound on long path constant in products with the line}
	Let $G$ be a finite connected graph. Let $p_{\star}=p_{\star}(G)$ be as above. Then
	\[p_{1, \ell \mathit{p}}(\mathbb{Z}\times G)  \leqslant p_{\star}.\]
\end{corollary}
\begin{proof}
	Apply Theorem~\ref{theorem: crossing G times Pn} with $\alpha=1/2$.
\end{proof}

\section{Imaginary limits of real constructions: a preliminary lemma}\label{section: preliminary lemma}
In this section we prove a lemma that we shall use in Sections~\ref{section: line lattice} and~\ref{section: complete graphs}.  The lemma will allow us to use certain vertex-based constructions to create other $1$-ipms that cannot be represented as vertex-based constructions (or would correspond to vertex-based constructions with `complex weights').

\begin{lemma}\label{lemma: preliminary lemma}
Let $G$ be a finite graph, and let $\mathcal{Q} := \{Q_{H}(\theta): H \subseteq G\}$ be a set of polynomials with real coefficients, indexed by subgraphs of $G$.  Given $\theta \in \mathbb{C}$, let $\mu_{\theta}$ be the following function from subgraphs of $G$ to $\mathbb{C}$:
\begin{equation}
\mu_{\theta}(H) := Q_{H}(\theta). \nonumber
\end{equation}
Suppose there exists a non-trivial interval $I \subseteq \mathbb{R}$ such that, for all $\theta \in I$, the function $\mu_{\theta}$ defines a $1$-ipm on $G$.  Suppose further that there exists a set $X \subseteq \mathbb{C}$ such that, for all $\theta \in X$ and all $H \subseteq G$, $\mu_{\theta}(H)$ is a non-negative real number.  Then $\mu_{\theta}$ is a $1$-ipm on $G$ for all $\theta \in X$.
\end{lemma}

\begin{proof}
We start by proving that $\mu_{\theta}$ is a measure on $G$ for all $\theta \in X$.  As $\mu_{\theta}(H)$ is a non-negative real number for all $\theta \in X$ and all $H \subseteq G$, all that is left to prove is that
\begin{equation}\label{equation: 1-ipm check 1}
\left(\sum_{H \subseteq G} Q_{H}(\theta) \right) - 1 = 0.
\end{equation}
The left hand side of (\ref{equation: 1-ipm check 1}) is a polynomial in $\theta$ with real coefficients, and is equal to zero for all $\theta$ in the interval $I$.  By the fact that a non-zero polynomial over any field has only finitely many roots,  the polynomial is identically zero and so (\ref{equation: 1-ipm check 1}) holds for all $\theta$.

We now show that $\mu_{\theta}$ is a $1$-ipm on $G$ for all $\theta \in X$.  To do this we must show that the following holds true for all $\theta \in X$,  for all $A,B \subseteq V(G)$ such that $A$ and $B$ are disjoint, and all $G_{1}$ and $G_{2}$ such that $G_{1}$ is a subgraph of $G[A]$ while $G_{2}$ is a subgraph of $G[B]$:
\begin{equation}\label{equation: 1-ipm check 2}
\mu_{\theta}\left(\mathbf{G}_{\mu_{\theta}} \left[A\right] = G_{1},\mathbf{G}_{\mu_{\theta}} \left[B\right] = G_{2}\right) = \mu_{\theta}\left(\mathbf{G}_{\mu_{\theta}}\left[A\right] = G_{1}\right) \mu_{\theta}\left( \mathbf{G}_{\mu_{\theta}} \left[B\right] = G_{2}\right).
\end{equation}
Both sides of (\ref{equation: 1-ipm check 2}) are polynomials in $\theta$ with real coefficients --- the left hand side, for example, can be written as
\begin{align*}
\sum_{H \subseteq G: \  H[A]=G_1, \ H[B]=G_2} Q_H(\theta).
\end{align*}  As $\mu_{\theta}$ is a $1$-ipm on $G$ for all $\theta \in I$, we have that these two polynomials agree on $I$, and so by the Fundamental Theorem of Algebra, they must be the same polynomial.  Thus (\ref{equation: 1-ipm check 2}) holds as required.
\end{proof}

\section{The line lattice $\mathbb{Z}$}\label{section: line lattice}
In this section we prove Theorem \ref{theorem: connected function Pn} on the connectivity function of paths.  Recall that, given $n \in \mathbb{N}_{\geqslant 2}$ and $p \in [0,1]$, we let $\theta = \theta(p) := \frac{1+\sqrt{4p-3}}{2}$ and $p_{n} := \frac{1}{4}\left(3 - \tan^{2}\left(\frac{\pi}{n+1} \right) \right)$.  Let $g_{n}(\theta):= \sum_{j = 0}^{n} \theta^{j}(1-\theta)^{n-j}$. 

We begin by constructing a measure $\nu_{p} \in \mathcal{M}_{1,\geqslant p}(P_{n})$ as follows.  Let us start with the case $p \geqslant \frac{3}{4}$.  For each vertex of $P_{n}$, we set it to state $0$ with probability $\theta$, and set it to state $1$ otherwise, and we do this independently for every vertex.  Recall that for each $j \in [n]$ we write $S_j$ for the state of vertex $j$; in this construction, the states are independent and identically distributed random variables. We set the edge $\{j,j+1\}$ to be open if $S_j \leqslant S_{j+1}$, and closed otherwise.  Thus, as $p = \theta + (1-\theta)^{2}$, we have that each edge is open with probability $p$.  Moreover $(\mathbf{P}_{n})_{\nu_{p}}$ will be connected if and only if there exists some $j \in [n+1]$ such that $S_{k} = 0$ for all $k < j$, while $S_{k} = 1$ for all $k \geqslant j$.  Therefore $(\mathbf{P}_{n})_{\nu_{p}}$ is connected with probability $g_{n}(\theta)$.  As this construction is vertex-based, it is clear that it is $1$-independent.

When $p < \frac{3}{4}$ we have that $\theta$ is a complex number, and so the above construction is no longer valid.  However, as discussed in Section \ref{section: preliminary lemma}, we will show that it is possible to extend this construction to all $p \in [p_{n},1]$.  For each subgraph $G$ of $P_{n}$, set $Q_{G}(\theta)$ to be the polynomial $\nu_{p}((\mathbf{P}_n)_{\nu_{p}}=G)$ for all $\theta \in [\frac{3}{4},1]$.  The following claim, together with Lemma \ref{lemma: preliminary lemma}, shows that in fact $\nu_{p}$ is a $1$-ipm on $P_{n}$ for all $p \in [p_{n},1]$. 
\begin{claim}\label{claim: Pn book keeping}
For all $p \in [p_{n},\frac{3}{4})$ and all $G \subseteq P_{n}$ we have that $Q_{G}\left(\theta\left(p\right)\right)$ is non-negative real number.
\end{claim}
\begin{proof}
We proceed by induction on $n$.  When $n = 2$ we have that there are only two possible subgraphs of $P_{2}$, which are $P_{2}$ itself and its complement $\overline{P_{2}}$.  We have that $Q_{P_{2}}(\theta(p)) = p$ and $Q_{\overline{P_{2}}}(\theta(p)) = 1-p$, so the claim holds as required for $n=2$.

Let us now assume that $n > 2$ and that the claim is true for all cases from $2$ up to $n-1$.  We split into two further subcases.  We first deal with the case that $G = P_{n}$.  We have that $Q_{P_{n}}\left(\theta\left(p\right)\right) = g_{n}(\theta)$.  For $p < \frac{3}{4}$ we can write
\begin{equation}\label{equation: g_n(theta)}
g_{n}(\theta) = \frac{\theta^{n+1} - (1-\theta)^{n+1}}{2\theta -1}.
\end{equation}
When $p < \frac{3}{4}$ we have that $\theta$ and $1-\theta$ are complex conjugates, and also that $2\theta - 1$ is a pure imaginary number.  Thus both the numerator and denominator of the above fraction are pure imaginary, and so $g_{n}\left(\theta\left(p\right)\right)$ is a real number for all $p < \frac{3}{4}$.  By writing $\theta = re^{i \phi}$, where $r := \sqrt{1-p}$ and $\phi := \arctan \left(\sqrt{3-4p}\right)$, we can rewrite (\ref{equation: g_n(theta)}) as
\begin{equation}\label{equation: g_n(theta) 2}
g_{n}\left(\theta\left(p\right)\right) = \frac{2 r^{n+1}}{\sqrt{3-4p}}\sin \left(\left(n+1\right) \phi \right). 
\end{equation}
Now $p\in [p_n, \frac{3}{4})$ implies $0 < \phi \leq \arctan \left(\sqrt{3-4p_n}\right)=\frac{\pi}{n+1}$, which in turn gives $\sin \left(\left(n+1\right) \phi \right)\geq 0$.
Thus by \eqref{equation: g_n(theta) 2} above, $g_{n}\left(\theta\left( p \right)\right)$ is a  non-negative real number for all $p$ in the interval $[p_{n},\frac{3}{4})$,  as required.

We now deal with the case that $G \neq P_{n}$.  Let us consider the vertex-based construction from which $Q_{G}(\theta)$ was defined.   As not every edge is present in $G$ we have that there exists some $j \in [n-1]$ such $\{j,j+1\}$ is not an edge, and so $S_{j} = 1$ while $S_{j+1} = 0$.  Note that if $j \geqslant 2$, then the edge $\{j-1,j\}$ is present in $G$ regardless of the state of vertex $j-1$.  Similarly, if $j \leqslant n-2$, then the edge $\{j+1,j+2\}$ is present in $G$ regardless of the state of vertex $j+2$.  If we write $G_{1} := G\left[\{1,\ldots,j-1\}\right]$ and $G_{2} := G\left[\{j+2,\ldots,n\}\right]$, then we have that
\begin{equation}\label{equation: polynomials}
Q_{G}(\theta) = \theta(1-\theta)Q_{G_{1}}(\theta)Q_{G_{2}}(\theta). 
\end{equation}
Now, by induction, we have that $Q_{G_{1}}\left( \theta\left( p \right)\right)$ and $Q_{G_{2}}\left( \theta \left( p \right) \right)$ are positive real numbers for all $p \in [p_{n},\frac{3}{4})$; to make this inductive step work we are using the fact that $(p_{n})_{n \geqslant 2}$ forms an increasing sequence, and so $p \geqslant p_{n}$ implies that $p \geqslant p_{s}$ for all $s \leqslant n$.  As $\theta\left(p\right)\left(1-\theta\left(p\right)\right) = 1-p$, we have that (\ref{equation: polynomials}) is a positive real for all $p \in [p_{n},\frac{3}{4})$, and so we have proven the claim.
\end{proof}

Note that as this proof shows that $g_{n}\left(\theta \left(p_{n}\right)\right) = 0$, we have that the probability $(\mathbf{P}_{n})_{\nu_{p_{n}}}$ is connected is equal to $0$.  As $\nu_{p_{n}} \in \mathcal{M}_{1,\geqslant p}(P_{n})$ for all $p \leqslant p_{n}$, we have that $f_{1,P_{n}}(p) = 0$ for all $p \leqslant p_{n}$.

We now prove that this construction is optimal with respect to the connectivity function.  Note that the following proof involves  essentially following the proof of Theorem \ref{theorem: crossing G times Pn} when $G$ consists of a single point and checking that the above construction is tight at every stage of this proof. Finally, we should emphasise that the main ideas in the construction of $\nu_p$ and its analysis are due to Balister and Bollob\'as~\cite{BalisterBollobas12} (they considered slightly different probabilities for vertex states, setting $S_k=0$ with probability $q_k$, where $q_k$ is defined for $k\in [n]$ by $q_1=0$ and  by the recurrence relation $q_k=\min\left(\frac{1-p}{1-q_{k-1}}, 1\right)$ for $k\geq 2$, which corresponds exactly to the equality case in inequality \eqref{equation: qn 2} below).

\begin{proof}[Proof of Theorem \ref{theorem: connected function Pn}]
The above construction discussed shows that
\begin{equation}
f_{1,P_{n}}(p) \leqslant \begin{cases} g_{n}(\theta) &\text{ for } p \in [p_{n},1] , \\
							  0 &\text{ for } p \in [0,p_{n}]. \end{cases} \nonumber
\end{equation}
It is clear that $f_{1,P_{n}}(p) \geqslant 0$ for all $p$, and so all that remains to show is that $f_{1,P_{n}}(p) \geqslant g_{n}\left(\theta\left(p\right)\right)$ for all $p \in [p_{n},1]$.

Let $\mu \in \mathcal{M}_{1,\geqslant p}(P_{n})$. For $k \in [n]$, let $X_{k}$ be the event that the subgraph of $(\mathbf{P}_{n})_{\mu}$ induced by the vertex set $[k]$ is connected, and let $H_{k}$ be the event that the edge $\{k-1,k\}$ is not present in $(\mathbf{P}_{n})_{\mu}$.  Applying random sparsification as in Remark~\ref{remark: random sparsification} if necessary, we may assume without loss of generality that for every $k$, the event $H_k$ occurs with probability exactly $1-p$.

Let $q_{2}^{\mu} := \mu\left((X_2)^c\right)=1-p$, and for $k > 2$ let $q_{k}^{\mu} := \mu\left(\left(X_{k}\right)^{c}| X_{k-1}\right)$.  We have that
\begin{eqnarray}
q_{k}^{\mu} &=& \frac{\mu(H_{k} \cap X_{k-1})}{\mu(X_{k-1})}\nonumber \\ 
& \leqslant &  \frac{\mu(H_{k} \cap X_{k-2})}{\mu(X_{k-1})} \label{equation: qn 1}\\ 
&=& \frac{\mu(H_{k})\mu(X_{k-2})}{\mu(X_{k-1})} \qquad \textrm{ by $1$-independence} \nonumber \\
&\leqslant & \frac{1-p}{1-q_{k-1}^{\mu}}.\label{equation: qn 2}
\end{eqnarray}
Note that $\mu((\mathbf{P}_{n})_\mu = P_{n}) = \prod_{j = 2}^{n}(1-q_{n}^{\mu})$.  Thus to show that the previous construction is optimal with respect to the connectivity function it is enough to show that equality holds for inequalities (\ref{equation: qn 1}) and (\ref{equation: qn 2}) when $\mu = \nu_{p}$.  In the measure $\nu_{p}$, we have that every edge is present with probability exactly $p$, thus $\nu_{p}(H_{k}) = 1-p$ and so equality holds in (\ref{equation: qn 2}).  To prove that equality holds in (\ref{equation: qn 1}), it is sufficient so show that
\begin{equation}\label{equation: qn 3}
\nu_{p}(H_{k} \cap X_{k-1}) = \nu_{p}(H_{k} \cap X_{k-2}).
\end{equation}
Both the left and right hand sides of (\ref{equation: qn 3}) can be expressed as polynomials in $\theta(p)$, and so it is sufficient to show that equality holds for $p \geqslant \frac{3}{4}$, as that will show they are the same polynomial (and so equality holds for all $p \in [p_{n},1])$.   Suppose that the event $(H_{k} \cap X_{k-2})$ occurs.  As $H_{k}$ has occurred we have that $S_{k-1} = 1$ while $S_{k} = 0$.  As $S_{k-1} = 1$, we have that edge $\{k-2,k-1\}$ is open, regardless of $S_{k-2}$.  Thus, as $ X_{k-2}$ has occurred we also have that $X_{k-1}$ has occurred.  Therefore $(H_{k} \cap X_{k-1})$ has also occurred, and so we are done.
\end{proof}

We remark in similar fashion to the above proof that the following holds for any $\mu \in \mathcal{M}_{1,\geqslant p}(P_{n})$:
\begin{eqnarray}
\mu(X_{n}) &\geqslant& \mu(X_{n-1}) - \mu(X_{n-1} \cap H_{n}) \nonumber \\
 &\geqslant& \mu(X_{n-1}) - \mu(X_{n-2} \cap H_{n}) \nonumber \\
&=& \mu(X_{n-1}) - \mu(X_{n-2})\mu(H_{n}^{c})  \qquad \textrm{ by $1$-independence}\nonumber \\
&\geqslant& \mu(X_{n-1}) - (1-p)\mu(X_{n-2}). \nonumber
\end{eqnarray}
Moreover, by once again considering what states of vertices can lead to the various events, we have that equality holds for all of the above inequalities when $\mu = \nu_{p}$.  This leads us to another way to define $g_{n}\left(\theta \left(p\right)\right)$: 
let $g_{1}\left(\theta \left(p \right) \right) :=1$,  $g_{2}\left(\theta \left(p\right)\right) := p$, and for all $n \geqslant 3$ we have that
\begin{equation}
g_{n}\left(\theta \left(p\right)\right) = g_{n-1}\left(\theta \left(p\right)\right) - (1-p)g_{n-2}\left(\theta \left(p\right)\right). \nonumber
\end{equation}

We conclude this section with a proof of Theorem~\ref{theorem: long paths critical prob}(i).

\begin{proof}[Proof of Theorem~\ref{theorem: long paths critical prob}(i)]
For the upper bound we plug $f_{1,P_1}(p)=1$ into equation~(\ref{equation:  pstar}), solve that equation to get $p_{\star}(P_1)=\frac{3}{4}$ and apply Corollary~\ref{corollary: bound on long path constant in products with the line} to obtain $p_{1, \ell \mathit{p}}(\mathbb{Z})\leqslant \frac{3}{4}$.

For the lower bound, let $p<\frac{3}{4}$ be fixed. As the sequence $(p_n)_{n\in \mathbb{N}}$ is monotone increasing and tends to $3/4$ as $n\rightarrow \infty$, there exists $N \in \mathbb{N}$ such that $p<p_N$. We showed in Theorem~\ref{theorem: connected function Pn} that there exists a measure $\nu_{p_N}\in \mathcal{M}_{1,\geqslant p_N}(P_N)$ such that the probability $\left(\mathbf{P}_N\right)_{\nu_{p_N}}$ is connected is equal to zero.

We use this measure to create a measure $\nu \in \mathcal{M}_{1,\geqslant p}(\mathbb{Z})$. For each $i\in \mathbb{Z}$, we let the subgraphs $\mathbf{\mathbb{Z}}_{\nu}[ (i(N-1)+[N])]$ on horizontal shifts of $P_N$ by $i(N-1)$ be independent identically distributed random variables with distribution given by $\nu_{p_N}$. This gives rise to a $1$-independent model $\nu$ on $\mathbb{Z}$ with edge-probability at least $p$ (in fact at least $p_N$). Furthermore, all connected components of $\mathbf{\mathbb{Z}}_{\nu}$ have size at most $2(N-1)-1$. In particular, $p_{1, \ell \mathit{p}}(\mathbb{Z})\geqslant p$. Since $p<\frac{3}{4}$ was chosen arbitrarily, this gives the required lower bound $p_{1, \ell  \mathit{p}}(\mathbb{Z})\geqslant \frac{3}{4}$.
\end{proof}

\section{The ladder lattice $\mathbb{Z}\times K_2$}\label{section: ladder}
In this section we construct a family of $1$-ipms on segments of the ladder $\mathbb{Z}\times P_2$ with edge-probability close to $2/3$ for which with probability $1$ there are no open left-right crossings. The idea of this construction is due to Walters and the second author~\cite{FalgasRavry12} (though the technical work involved in rigorously showing the construction works is new).

Let us begin by giving an outline of our construction. We write the vertex set $V(P_{N} \times P_{2})$ as $[N] \times [2]$.  As in the case of the line lattice, we independently assign to each vertex $(n,y)$ a random state $S_{(n,y)}$.  If $n+y$ is even, then we let
	\begin{align*}
	S_{(n,y)}:=\begin{cases}
	2 & \textrm { with probability }p_n,\\
	0 & \textrm { with probability }1-p_n;
	\end{cases}
	\end{align*} 
	while if instead $n+y$ is odd, then we let
	\begin{align*}
	S_{(n,y)}:=\begin{cases}
	2 & \textrm { with probability }r_n,\\
	1 & \textrm{ with probability }s_n,\\
	0 & \textrm { with probability }1-r_n-s_n.
	\end{cases}
	\end{align*}
	Here $(p_n)_{n\in \mathbb{N}}, (r_n)_{n\in \mathbb{N}}, (s_n)_{n\in \mathbb{N}}$ are suitably chosen sequences of real numbers, ensuring that the $S_{(n,y)}$ are well-defined random variables. We then define a random spanning subgraph $\mathbf{G}_{\mu}$ of $G:=P_N\times P_2$ from the random vertex states $S_{(n,y)}$: $(n,y)\in [N]\times [2]$ as follows:
	\begin{itemize}
		\item for each $n\in [N-1]$ and $y\in [2]$, the horizontal edge $\{(n,y),(n+1, y)\}$ is open in $\mathbf{G}_{\mu}$ if and only if $S_{(n,y)}\leqslant S_{(n+1, y)}$,
		\item for each $n\in [N]$, the vertical edge $\{(n,1), (n,2)\}$ is open in $\mathbf{G}_{\mu}$ if and only if $\left(S_{(n,1)}-S_{(n,2)}\right)(1-S_{(n,1)})(1-S_{(n,2)})=0$.
	\end{itemize}
Note the condition for a vertical edge $\{(n,1), (n,2)\}$ to be open can be rephrased as if and only if either $S_{(n,1)}=S_{(n,2)}$ or one of $S_{(n,1)}, S_{(n,2)}$ is equal to $1$. So intuitively, the value of the $S_{(n,y)}$ must increase from left to right along open horizontal edges, and it must stay constant along open vertical edges unless one of the endpoints is in the special state $1$, which allows free passage up or down.

Clearly the bond percolation measure $\mu$ associated to our  random graph model $\mathbf{G}_{\mu}$ is a $1$-ipm on the ladder $G=P_N\times P_2$ as it is vertex-based. By making a judicious choice of the sequences  $(p_n)_{n\in \mathbb{N}}, (r_n)_{n\in \mathbb{N}}, (s_n)_{n\in \mathbb{N}}$ and taking $N$ sufficiently large, one can ensure that in addition $\mu$ satisfies  $d(\mu)\geq p$ and $\mu(\exists\textrm{ open left-right crossing})=0$. In particular, with this construction we prove the following result.
\begin{theorem}\label{theorem: ladder construction}
	Fix $p\in (\frac{1}{2}, \frac{2}{3})$. Then there exists $N\in \mathbb{N}$ such that for all $n \geqslant N$,
	\[p_{1, \times}(P_n\times P_2)\geqslant p.\]
\end{theorem}	
\begin{proof}
Fix $p:=\frac{2}{3}-\varepsilon$, with $\varepsilon \in (0,\frac{1}{6})$. We start by defining the sequences $\left(p_n\right)_{n \in \mathbb{N}}$, $\left(r_n\right)_{n\in \mathbb{N}}$ and $\left(s_n\right)_{n \in \mathbb{N}}$ iteratively as follows. We set $p_1=r_1=1$ and $s_1=0$. Then for $n \in \mathbb{N}$, we let
\begin{align*}
p_{n+1}&=\left\{\begin{array}{ll}
1-\frac{1-p}{r_n+s_n} & \textrm{if }r_n+s_n\geqslant 1-p,\\
0 & \textrm{otherwise;}
\end{array}\right. \\
r_{n+1}&=\left\{\begin{array}{ll}
1-\frac{1-p}{p_n} & \textrm{if }p_n\geqslant 1-p,\\
0 & \textrm{otherwise;}
\end{array}\right. \\
s_{n+1}&=\left\{\begin{array}{ll}
\max\left\{1-2r_{n+1}+\frac{r_{n+1}-(1-p)}{p_{n+1}}, 0 \right\} & \textrm{if }p_{n+1}>0,\\
0 & \textrm{otherwise.}
\end{array}\right.
\end{align*}
\begin{lemma}\label{lemma: sequences work}The following hold for all $n\in \mathbb{N}$:
	\begin{enumerate}[(i)]
		\item $p_n, r_n\in [0,1]$,
		\item $s_n \in [0, 1-r_n]$,
		\item $p_{n+1}\leqslant p_n$,
		\item  $r_{n+1}\leqslant r_n$,
		\item $r_{n+1}+s_{n+1}\leqslant r_n+s_n$.
	\end{enumerate}		
\end{lemma}	
\begin{proof}
	We prove the lemma by induction on $n$. By definition of our sequences, $p_1=r_1=1\geqslant p=p_2=r_2$, $s_1=0$, and $0<s_2=\frac{(2p-1)(1-p)}{p}< 1-p=r_1+s_1-r_2$, and thus (i)--(v) all hold in the base case $n=1$. 

	Suppose now (i)--(v) hold for all $n \leqslant N$, for some $N \geqslant 1$. Since $p_N$ and $r_N+s_N$ both lie in $[0,1]$, the definition of $p_{N+1}$ and $r_{N+1}$ implies these also both lie in $[0,1]$. This establishes (i) for $n=N+1$. By construction, $s_{N+1}\geqslant 0$, and by the inductive hypotheses (ii) and (v), we have
	\begin{align*}
	s_{N+1}\leqslant r_{N}+s_N-r_{N+1}\leqslant 1-r_{N+1},
	\end{align*}
	whence (ii) holds for $n=N+1$.

	If $p_{N+2}=0$, then $p_{N+2}\leqslant p_{N+1}$ trivially holds (since $p_{N+1}\geqslant 0$ by (i)). On the other hand, suppose $p_{N+2}=1-\frac{1-p}{r_{N+1}+s_{N+1}}>0$. Then we have $r_{N+1}+s_{N+1}> 1-p$, which by our inductive hypothesis (v) implies $r_N+s_N\geqslant r_{N+1}+s_{N+1}> 1-p$. The definition of $p_{N+1}$ then implies
	\begin{align*}
	p_{N+2}=1-\frac{1-p}{r_{N+1}+s_{N+1}} \leqslant 1-\frac{1-p}{r_{N}+s_{N}}=p_{N+1},
	\end{align*}
	as desired, establishing that (iii) holds for $n=N+1$. Arguing in exactly the same way (using the inductive hypothesis (iii) instead of (v)), we obtain that $r_{N+2}\leqslant r_{N+1}$.  Hence (iv) holds for $n=N+1$.
	
	Finally we consider (v) for $n=N+1$, which is the most delicate part of the induction.  We begin by recording two useful facts, the second of which we shall reuse later.
		\begin{claim}\label{claim: if p_{n+2}=0 or r_{n+2}=0 then s_{n+2}=0}
			If $p_{N+2}=0$ or $r_{N+2}=0$, then $s_{N+2}=0$.
		\end{claim}	
		\begin{proof}
			If $p_{N+2}=0$, then by construction $s_{N+2}=0$ and so we are done. If $r_{N+2}=0$, then by construction $p_{N+1}\leq 1-p$, which by our inductive hypothesis (iii) implies $p_{N+2}\leq 1-p$ and hence $s_{N+2}= \max\left\{1-2r_{N+2}+ \frac{r_{N+2}-(1-p)}{P_{N+2}}, 0 \right\}=\max\left\{1- \frac{1-p}{p_{N+2}}, 0\right\}=0$.
		\end{proof}	
	\begin{claim}\label{claim: intermediate result --- s_i >0 for all smaller i}
		If $p_{N+2}$ and $r_{N+2}$ are both strictly positive, then for all $i\in \{2,\ldots,N+1\}$, we have  $s_i>0$.
	\end{claim} 
	\begin{proof}
		Fix $i\in \{2,\ldots,N+1\}$.  By our inductive hypotheses (iii)--(iv) (which we have already established up to $n=N+1$) and since $i\geqslant 2$, we have  $0 < p_{N+2}\leqslant p_i\leqslant p_2=p$ and $0 < r_{N+2} \leqslant r_i\leqslant r_2=p$. Since $r_{i+1}>0$, we in fact have $p_{i}>1-p$. We also have that
		\begin{align*}
		1-2r_i+\frac{r_i-(1-p)}{p_i}=\frac{1}{p_i}\left(p_i +r_i -2r_ip_i -(1-p)\right)=:\frac{1}{p_i} f(p_i, r_i).
		\end{align*}
		Now for fixed $y\in [1/2, 1]$, the function $x\mapsto f(x,y)$ is a non-increasing function of $x$. Thus if $r_i\geqslant 1/2$, we have
		\begin{align*}
		f(p_i, r_i)\geqslant f(p, r_i)= (2p-1)(1-r_i)>0.
		\end{align*}
		On the other hand for fixed $y\in (0,1/2)$, the function $x\mapsto f(x,y)$ is strictly increasing in $x$. Therefore if $r_i<1/2$, we have
		\begin{align*}
		f(p_i, r_i)> f(1-p, r_i)= r_i(2p-1)>0. 
		\end{align*}
		In either case, $f(p_i, r_i)>0$, and thus $s_i=\max\left(\frac{1}{p_i}f(p_i, r_i), 0 \right)>0$.
	\end{proof}
With these results in hand, we return to the proof of (v). If $s_{N+2}=0$,  then (v) follows immediately from (iv). Thus we may assume that $s_{N+2}>0$, whence by Claim~\ref{claim: if p_{n+2}=0 or r_{n+2}=0 then s_{n+2}=0} $p_{N+2}>0$ and $r_{N+2}>0$. By Claim~\ref{claim: intermediate result --- s_i >0 for all smaller i} and our inductive hypotheses (iii) and (iv), this implies that $p_i$, $r_i$ and $s_i$ are all strictly positive for $i\in \{2, 3 \ldots, N+2\}$.
By definition of our sequences we thus have for all $i \in [N+1]$ that 
	\begin{align}\label{eq: expressions for pi, ri, si}
	p_{i+1}=1-\frac{1-p}{r_{i}+s_{i}},&&r_{i+1}=1-\frac{1-p}{p_{i}},
	&&s_{i+1}&=1-2r_{i+1}+\frac{r_{i+1}-(1-p)}{p_{i+1}}.
	\end{align}
	Combining these equations we obtain for $i\in \{2,\ldots,N+1\}$ that:
	\begin{align}\label{eq: p_{i+1} as a function of p_i, p_{i-1}}
	p_{i+1}=1-\frac{p_{i}p_{i-1}(1-p)}{p(p_{i-1}-p_{i}+1)+p_{i}-1}.
	\end{align}
	\begin{claim}\label{claim: intermediate result --- p_{i+1} as function of p_i for all smaller i}
	Under our assumption that $s_{N+2}>0$, for all	integers $i\in [N+1]$ we have 
		\begin{align*}
	p_{i+1}=\frac{p_{i}-(1-p)}{(2-p)p_{i}-(1-p)}.
	\end{align*}
\end{claim} 
\begin{proof}
Since $p_1=1$ and $p_2=p$, our claim holds for $i=1$. Suppose it holds for some $i\leqslant N$. Then by rearranging terms, we have
\begin{align*}
p_{i}= \frac{(1-p)(1-p_{i+1})}{1- (2-p)p_{i+1}}.
\end{align*}
Substituting this into the formula for $p_{i+2}$ given by (\ref{eq: p_{i+1} as a function of p_i, p_{i-1}}), we see our claim holds for $i+1$ as well.
\end{proof}
It follows from Claim~\ref{claim: intermediate result --- p_{i+1} as function of p_i for all smaller i} and (\ref{eq: expressions for pi, ri, si}) that for all $i\in [N+1]$, we can write $r_{i+1}+s_{i+1}$ as a function $p_{i}$:

\begin{eqnarray}\label{eq: t_{i+1 formula}}
r_{i+1}+s_{i+1}&=& 1-r_{i+1}+\frac{r_{i+1}-(1-p)}{p_{i+1}} \nonumber \\
&=&\frac{p_{i}p(2-p)-(1-p)}{p_{i}-(1-p)}\nonumber \\
&=&p(2-p)-\frac{(1-p)^3}{p_i-(1-p)}.
\end{eqnarray}
For $p_{i}>(1-p)$ (which we recall holds since $r_{i+1}>0$), the expression above is an increasing function of $p_{i}$.  By our inductive hypothesis (iii) that $p_{N+1} \leqslant p_{N}$ it follows that $r_{N+2}+s_{N+2} \leqslant r_{N+1}+s_{N+1}$ and we have verified that (v) holds for $n=N+1$.
\end{proof}
\noindent Recall that $p= \frac{2}{3}-\varepsilon$, for some fixed $\varepsilon \in (0, \frac{1}{6})$. 
\begin{lemma}\label{lemma: ladder construction is 2-inaccessible if suff far}
	We have that $p_n=r_n=s_n=0$ for all $n\geqslant N_{\varepsilon}$, where $N_{\varepsilon}:= \lceil 2\varepsilon^{-1}\rceil$.
\end{lemma}	
\begin{proof}
Suppose first that there exists $m\in [N_{\varepsilon}-1]$ such that $r_{m}=0$.  Then $p_{m+1}=0$ and $s_{m+1}=0$ by construction and $r_{m+1}=0$ by Lemma~\ref{lemma: sequences work}(iv). Lemma~\ref{lemma: sequences work}(iii)-(v) then implies $p_{n}=r_{n}=s_{n}=0$ for all $n\in \mathbb{N}_{\geqslant m+1}$, as required.
	
Suppose instead that $r_{n} > 0$ for all $n \in [N_{\varepsilon}-1]$ and there exists some $m\in [N_{\varepsilon}-2]$ such that $p_{m}\leqslant 1-p$.  Then $r_{m+1}=0$, and thus by the argument above, we have that $p_{n}=r_{n}=s_{n}=0$ for all $n\in \mathbb{N}_{\geqslant m+2}$, as required.
	
Finally, suppose $p_n>1-p$ and $r_n>0$ both hold for all $n \in [N_{\varepsilon}-2]$. By 
Claim~\ref{claim: intermediate result --- s_i >0 for all smaller i},
 we have $s_n>0$ for all $n\in \{2,\ldots,N_{\varepsilon}-3\}$. This allows us in turn to apply Claim~\ref{claim: intermediate result --- p_{i+1} as function of p_i for all smaller i} to all $n$ in this interval and to deduce that
	\begin{eqnarray}\label{ineq: p_n -p_{n-1} less than -eps}
	p_{n-1}- p_{n}&=&p_{n-1}-\frac{p_{n-1}-(1-p)}{(2-p)p_{n-1}-(1-p)} \nonumber\\
	&=&\frac{1}{(2-p)p_{n-1}-(1-p)}\left((2-p) \left(p_{n-1}-\frac{1}{2}\right)^2+\frac{2-3p}{4} \right) \nonumber \\
	&\geqslant& \frac{3\varepsilon}{4}.
	\end{eqnarray}
Recall that $p_1=1$.  As such, it follows from inequality~(\ref{ineq: p_n -p_{n-1} less than -eps}) that $p_{n} \leqslant 1 - (n-1)\frac{3\varepsilon}{4}$ for all $n\in [N_{\varepsilon} -2 ]$.  In particular, as $N_{\varepsilon}=\lceil 2 \varepsilon^{-1}\rceil $ and $\varepsilon \in (0, \frac{1}{6})$, we have 
\[p_{N_{\varepsilon}-3}\leq 1- \left(\frac{2}{\varepsilon}-4\right)\frac{3\varepsilon}{4}=-\frac{1}{2}+3\varepsilon< \frac{1}{3}+\varepsilon=1-p,\]
which is a contradiction.
\end{proof}

Now let $N=N_{\varepsilon}$ be the integer constant whose existence is given by Lemma~\ref{lemma: ladder construction is 2-inaccessible if suff far} and construct the $1$-ipm $\mathbf{G}_\mu$ on the graph $G=P_{N}\times P_2$ from independent random assignments of states $S_{(n,y)}$ to  vertices $(n,y)$ in $V(G)=[N]\times [2]$, as described at the beginning of this section.

We observe here that  by Lemma~\ref{lemma: sequences work}(i)--(ii), the states $S_{(n,y)}$ are well-defined random variables for every $(n,y)\in [N]\times [2]$, and so $\mu$ is a well-defined $1$-ipm. We recall here for the reader's convenience the state-based rules governing which edges are open in $\mathbf{G}_{\mu}$:
	\begin{itemize}
		\item for each $n\in [N-1]$ and $y\in [2]$, the horizontal edge $\{(n,y),(n+1, y)\}$ is open if and only if $S_{(n,y)}\leqslant S_{(n+1, y)}$,
		\item for each $n\in [N]$, the vertical edge $\{(n,1), (n,2)\}$ is open if and only if either $S_{(n,1)}=S_{(n,2)}$ or one of $S_{(n,1)}, S_{(n,2)}$ is equal to $1$.
	\end{itemize}
So intuitively, the value of the $S_{(n,y)}$ must increase from left to right along open horizontal edges of $\mathbf{G}_{\mu}$, and it must stay constant along open vertical edges of $\mathbf{G}_{\mu}$ unless one of the endpoints is in the special state $1$ which allows free passage up or down. 
	\begin{claim}\label{claim: ladder measure has right edge probability}
		We have that $d(\mu)\geqslant p$.	
	\end{claim}
	\begin{proof}
		For $(n,y)\in [N-1]\times [2]$, consider the horizontal edge $\{(n,y), (n+1,y)\}$, . If $n+y$ is even, then by definition of $r_{n+1}$,
		\begin{align*}
		\mu\left(\{(n,y), (n+1,y)\}\in \mathbf{G}_\mu \right)&=\mu \left(S_{(n,y)}\leqslant S_{(n+1,y)}\right)=r_{n+1}+(1-r_{n+1})(1-p_n)\geqslant p.
		\end{align*} 	
		Similarly if $n+y$ is odd, then by definition of $p_{n+1}$,
		\begin{align*}
		\mu \left(\{(n,y), (n+1,y)\}\in \mathbf{G}_\mu \right)&=\mu \left(S_{(n,y)}\leqslant S_{(n+1,y)}\right)=p_{n+1}+(1-p_{n+1})(1-r_n-s_n)\geqslant p.
		\end{align*} 
		Finally, for a vertical edge $\{(n,1), (n,2)\}$, $n\in [N]$, we have
		\begin{align*}
		\mu \left(\{(n,1), (n,2)\}\in \mathbf{G}_\mu \right)&=\mu \left(S_{(n,1)}=S_{(n,2)} \textrm{ or } 1\in \{S_{(n,1)}, S_{(n,2)} \}\right)\\
		&=s_n + p_{n}r_{n}+ (1-p_{n})(1-r_{n}-s_{n}).
		\end{align*}
		Now, if $p_n=0$, then $r_{n-1}\leqslant 1-p$ by definition of $p_n$, whence $r_n\leqslant 1-p$ by Lemma~\ref{lemma: sequences work}(iv), and so the expression above equals $1-r_n \geqslant p$. On the other hand if $p_n\neq 0$, then by definition of $s_n$ the expression above is at least $p$. Thus each horizontal edge and each vertical edge is open in $\mathbf{G}_\mu$ with probability at least $p$, and $d(\mu)\geqslant p$ as claimed.
	\end{proof}
	\begin{claim}
		There is no open path in $\mathbf{G}_\mu$ from $\{1\}\times [2]$ to $\{N\}\times [2]$.
	\end{claim}
	\begin{proof}
		By construction, $p_1=r_1=1$, whence $S_{(1,1)}=S_{(1,2)}=2$. Furthermore, by Lemma~\ref{lemma: ladder construction is 2-inaccessible if suff far} and our choice of $N$, $p_N=r_N=s_N=0$, whence $S_{(N,1)}=S_{(N,2)}=0$.

		Let $N'$ be the largest $n\in [N]$ for which there exists an open path in $\mathbf{G}_\mu$ from $\{1\}\times [2]$ to $\{n\}\times [2]$. Let $\mathcal{P}$ be such a path, and let $v_0\in \{1\}\times [2]$, $v_1\in \{2\}\times[2]$, $v_2$, ... , $v_{\ell} \in \{n\}\times [2]$ be the vertices of $\mathcal{P}$ traversed from left to right. Observe that in this ordering of the vertices of $\mathcal{P}$, every horizontal edge $\{(n,y),(n+1,y)\}$ of $\mathcal{P}$ is traversed from left to right.

		We claim that for all $i\in [\ell]$, we have $S_i\in\{1,2\}$. Indeed, by construction $S_{v_0}=2$. Suppose there exists some $1\leqslant i<\ell$ such that $S_{v_j}\in\{1,2\}$ for all $j < i$. If $S_{v_i}=2$, then the edge $v_iv_{i+1}$ can be open in $\mathbf{G}_\mu$ only if $S_{v_{i+1}}\in \{1,2\}$. What is more, $S_{v_{i+1}}$ can be equal to $1$ if and only if $v_iv_{i+1}$ is a vertical edge. On the other hand, suppose $S_{v_i}=1$. Then $v_{i-1}v_i$ was a vertical edge (since there is no edge both of whose endpoints are in state $1$ and since horizontal edges are traversed from left to right by $\mathcal{P}$), and so $v_{i+1}=v_i+(1,0)$. But then $v_iv_{i+1}$ open in $\mathbf{G}_\mu$ implies $S_{v_{i+1}}=2$. Thus for every vertex $v_i$ of $\mathcal{P}$, we have that $S_{v_i}$ is indeed in state $1$ or $2$.

		This implies in particular that $v_{\ell}\notin \{N\}\times [2]$ (since as we remarked above $S_{(N,1)}=S_{(N,2)}=0$). Thus there is no open path in $\mathbf{G}_\mu$ from $\{1\}\times [2]$ to $[N]\times [2]$.
	\end{proof}
	Thus $\mu$ is an element of $\mathcal{M}_{1,\geqslant p}(P_N\times [2])$ for which
	\[\mu(\exists\textrm{ open left-right crossing})=0.\]
	Given $n\geqslant N$, we may extend $\mu$ to an element $\mu'\in \mathcal{M}_{1,\geqslant p}(P_n\times P_2)$ by letting every edge in $P_n\times P_2\setminus P_N\times P_2$ be open independently at random with probability $p$. In this way we obtain a $1$-independent bond percolation measure $\mu'$ on $P_n\times P_2$ with edge-probability $p$ for which there almost surely are no open left--right crossings of $P_n\times P_2$, giving the required lower bound on $p_{1, \times}(P_n\times P_2)$.
\end{proof}	

We conclude this section by proving  Theorem~\ref{theorem: long paths critical prob}(ii), with the aid of Theorem~\ref{theorem: ladder construction}.

\begin{proof}[Proof of Theorem~\ref{theorem: long paths critical prob}(ii)]
	Trivially, the $1$-independent connectivity function of the path on $2$ vertices $P_2$ (i.e. the graph consisting of a single edge) is $f_{1,P_2}(p)=p$. Thus the constant $p_{\star}(P_2)$ defined by equation~(\ref{equation:  pstar}) is the unique solution in $[0,1]$ to the equation $x^2=4(1-x)^2$, namely $p_{\star}(P_2)=\frac{2}{3}$. By Corollary~\ref{corollary: bound on long path constant in products with the line}, this implies $p_{1, \ell \mathit{p}}(\mathbb{Z}\times P_2)\leqslant \frac{2}{3}$.

	For the lower bound, fix $p \in (\frac{1}{2},\frac{2}{3})$. In the proof of Theorem~\ref{theorem: ladder construction}, we showed there exist some integer $N\in \mathbb{N}$ and $\mu \in \mathcal{M}_{1,\geqslant p}(P_N\times P_2)$ such that 
	\begin{enumerate}[(i)]
		\item $\mu(\exists\textrm{ open left-right crossing})=0$;
		\item $\mu\left(\{(1,1), (1,2)\} \textrm{ and } \{(N,1), (N,2)\} \textrm{ are open}\right)=1$.
	\end{enumerate}	 
	We use this measure to create a measure $\nu \in \mathcal{M}_{1,\geqslant p}(\mathbb{Z} \times P_2)$. Let $G:=\mathbb{Z} \times P_2$. For each $i\in \mathbb{Z}$, we let the subgraphs $\mathbf{G}_{\nu}\left[ (i(N-1)+[N])\times [2]\right]$ on horizontal shifts of the ladder $P_N\times P_2$ by $i(N-1)$ be independent identically distributed random variables with distribution given by $\mu$. Thanks to property (ii) recorded above, the random subgraphs agree on the vertical rungs $\{1+i(N-1)\}\times P_2$ of the ladder, and this gives rise to a bona fide $1$-independent model $\nu$ on $\mathbb{Z} \times P_2$ with edge-probability $p$. Furthermore, property (i) implies all connected components in $\mathbf{G}_{\nu}$ have size at most $4(N-1)-2=4N-6$. In particular, $p_{1, \ell \mathit{p}}(\mathbb{Z}\times P_2)\geqslant p$. Since $p<\frac{2}{3}$ was chosen arbitrarily, this gives the required lower bound $p_{1, \ell \mathit{p}}(\mathbb{Z}\times P_2)\geqslant \frac{2}{3}$.
\end{proof}

\section{Complete graphs}\label{section: complete graphs}
In this section we will prove Theorem \ref{theorem: connected function Kn}.    Recall that, given $ n \in \mathbb{N}_{\geqslant 2}$ and $p \in [0,1]$, we let $\theta = \theta(p) := \frac{1 + \sqrt{2p - 1}}{2}$ and $p_{n} := \frac{1}{2}(1 - \tan^{2}(\frac{\pi}{2n}))$.  Let $g_{n}(\theta) := \theta^{n} + (1-\theta)^{n}$.

\subsection{An upper bound for $f_{1, K_n}(p)$}
Before proving Theorem \ref{theorem: connected function Kn}, let us give a simple vertex-based construction of a measure $\nu_{p} \in \mathcal{M}_{1,\geqslant p}(K_{n})$ that shows $f_{1,K_{n}}(p) \leqslant g_{n}(\theta)$ for $p \geqslant \frac{1}{2}$.  We call this measure the \textit{Red-Blue construction}.  We think of $K_{n}$ as the complete graph on vertex set $[n]$, and we colour each vertex Red with probability $\theta$ and colour it Blue otherwise, and we do this independently for all vertices.  The edge $\{i,j\} \in [n]^{(2)}$ is open if and only if $i$ and $j$ have the same colour.  As $p = \theta^{2} + (1- \theta)^{2}$, we have that each edge is present in $(\mathbf{K}_{n})_{\nu_{p}}$ with probability $p$.  Note that $(\mathbf{K}_{n})_{\nu_{p}}$ will either be either a disjoint union of two cliques, in which case it is disconnected, or the complete graph $K_{n}$, in which case it is connected.  This latter case occurs if and only if every vertex receives the same colour, and so the probability that $(\mathbf{K}_{n})_{\nu_{p}}$ is connected is equal to $g_{n}(\theta)$.  As this construction is vertex-based, it is clear that it is $1$-independent.

If $p < \frac{1}{2}$ then $\theta$ is a complex number, and so the Red-Blue construction is no longer valid.  However, as discussed in Section \ref{section: preliminary lemma}, we will show that it is possible to extend this construction to all $p \in [p_{n},1]$.  Given $j \in \{0,1,\ldots,n\}$, let 
\begin{equation}
g_{n,j}(\theta) := \theta^{j}(1-\theta)^{n-j} + \theta^{n-j}(1-\theta)^{j}. \nonumber
\end{equation}
When $j = 0$ or $j = n$ we have that $g_{n,0}(\theta)$ and $g_{n,n}(\theta)$ are each equal to $g_{n}(\theta)$,  and so we just write the latter instead.  Given some $A \subseteq [n]$, let $H_{A}$ be the disjoint union of a clique on $A$ with a clique on $[n] \setminus A$.  Note that when $A = \emptyset$ or $[n]$ we have that $H_{A}$ is equal to $K_{[n]}$, and more generally that $H_{A} = H_{[n]\setminus A}$. For $p \in [0,1]$, let $\mu_{p}$ be the following function on subgraphs $G$ of $K_{n}$:
\begin{equation}
\mu_{p}\left( G \right) := \begin{cases} g_{n,|A|}\left( \theta\left(p \right)\right) &\text{ if } G = H_{A} \text{ for some } A \subseteq [n] , \\
0 &\text{ else}. \end{cases} \nonumber
\end{equation}
For $p \in [\frac{1}{2},1]$ this function matches the Red-Blue construction given above, and so by defining $\nu_{p}((\mathbf{K}_n)_{\nu_p}=G):=\mu_p(G)$ for all subgraphs $G \subseteq K_n$, we obtain a measure $\nu_p$ which is a $1$-ipm defined without making reference to states of vertices.  The following claim, together with Lemma \ref{lemma: preliminary lemma}, shows that in fact $\nu_{p}$ is a $1$-ipm on $K_{n}$ for all $p \in [p_{n},1]$. 
\begin{claim}\label{claim: Kn book keeping}
For all $p \in [p_{n},\frac{1}{2}]$ and all $j \in \{0,\ldots,n\}$ we have that $g_{n,j}\left(\theta\left(p\right)\right)$ is non-negative real number.
\end{claim}
\begin{proof}
Let us begin with the case $j = n$. As $p \leqslant \frac{1}{2}$, we have that $\theta$ and $1-\theta$ are complex conjugates, and so $g_{n}\big(\theta(p)\big)$ is a real number for all $p$ in this range.  By writing $\theta = re^{i \phi}$, where $r := \sqrt{\frac{1-p}{2}}$ and $\phi := \arctan \left(\sqrt{1-2p}\right)$, we can write
\begin{equation}\label{eq: Kn g_n formula}
g_{n}\left(\theta\left(p\right)\right) = 2 r^{n} \cos \left(n \phi \right).  
\end{equation}
Now $p\in [p_n, \frac{1}{2}]$ implies $0\leq \phi \leq \frac{\pi}{2n}$, which in turn gives $\cos(n\pi)\geq 0$. By~\eqref{eq: Kn g_n formula}, it follows that 
$g_{n}\left(\theta\left(p\right)\right)$ is a non-negative real number for all $p \in [p_{n},\frac{1}{2}]$, which proves the claim when $j = n$.
For general $j \in \{0,\ldots,n\}$, we have that
\begin{equation}\label{equation: Kn book keeping}
g_{n,j}\left(\theta\left(p \right)\right) = \begin{cases}
\left(\frac{1-p}{2}\right)^{j} g_{n-2j} \left(\theta \left(p \right)\right)&\text{ if }  n \geqslant 2j,\\
\left(\frac{1-p}{2}\right)^{n-j} g_{2j-n} \left(\theta \left(p \right)\right)&\text{ if } n \leqslant 2j. \end{cases} \nonumber
\end{equation}
Therefore the previous case of the claim shows that $g_{n,j}\left(\theta\left(p \right)\right) \in [0,1]$ for all $p \in [p_{n},\frac{1}{2}]$; at this stage we are using the fact that $(p_{n})_{n \geqslant 2}$ forms an increasing sequence, and so $p \geqslant p_{n}$ implies that $p \geqslant p_{s}$ for all $s \leqslant n$.
\end{proof}
Note that as this proof shows that $g_{n}\left(\theta \left(p_{n}\right)\right) = 0$, we have that the probability $(\mathbf{K}_{n})_{\nu_{p_{n}}}$ is connected is equal to $0$.  As $\nu_{p_{n}} \in \mathcal{M}_{1,\geqslant p}(K_{n})$ for all $p \leqslant p_{n}$, we have that $f_{1,K_{n}}(p) = 0$ for all $p \leqslant p_{n}$.  We now prove that this construction is optimal with respect to the connectivity function.

\subsection{A lower bound on $f_{1,K_n}(p)$}
\begin{proof}[Proof of Theorem \ref{theorem: connected function Kn}]
The previous constructions discussed show that
\begin{equation}
f_{1,K_{n}}(p) \leqslant \begin{cases} g_{n}(\theta) &\text{ for } p \in [p_{n},1] , \\
							  0 &\text{ for } p \in [0,p_{n}]. \end{cases} \nonumber
\end{equation}
It is clear that $f_{1,K_{n}}(p) \geqslant 0$ for all $p$, and so all that remains to show is that $f_{1,K_{n}}(p) \geqslant g_{n}(\theta)$ for $p \in [p_{n},1]$.  We will prove this result by induction on $n$.  The inequality is trivially true when $n = 2$, so let us assume that $n > 2$ and that the inequality is true for all cases from $2$ up to $n-1$.  First, we note that $g_{n}(\theta) = g_{j}(\theta)g_{n-j}(\theta) - g_{n,j}(\theta)$ for all $j \in \{0,1,\ldots,n\}$.  Thus, if we multiply both sides of this equation by $\binom{n}{j}$ and sum over all $j \in \{0,1,\ldots,n\}$, we have that
\begin{equation}\label{equation: f-equation}
2^{n}g_{n}(\theta) = \Bigg( \sum_{j = 0}^{n}\binom{n}{j}g_{j}(\theta)g_{n-j}(\theta)\Bigg) - 2,
\end{equation}
Let $\mu \in \mathcal{M}_{1,\geqslant p}(K_{n})$ and let $C$ be the event that $(\mathbf{K}_n)_{\mu}$ is connected.  Given $A \subseteq [n]$, let $X_{A}$ be the event that $(\mathbf{K}_n)_{\mu}[A]$ and $(\mathbf{K}_n)_{\mu}[A^{c}]$ are each connected, where $A^{c} = 	[n] \setminus A$.  Moreover, let $Y_{A}$ be the event that  $(\mathbf{K}_n)_{\mu}[A]$ and $(\mathbf{K}_n)_{\mu}[A^{c}]$ are each connected, and there are no edges between $A$ and $A^{c}$ in $(\mathbf{K}_n)_{\mu}$.  For all $A \subseteq [n]$, we have that
\begin{equation}\label{inequality: f-inequality2}
\mu(C) \geqslant \mu(X_{A}) - \mu(Y_{A}).
\end{equation}
Note that when $A = \emptyset$ or $A = [n]$, the above equation is trivially true due to the fact that $C, X_{\emptyset},X_{[n]},Y_{\emptyset}$ and $Y_{[n]}$ are all the same event.  As $\mu$ is $1$-independent we have that if $A$ is a non-empty proper subset of $[n]$, then, by induction on $n$, we have
\begin{equation}\label{inequality: f-inequality1}
\mu(X_{A}) \geqslant g_{|A|}(\theta)g_{n - |A|}(\theta).
\end{equation}
Note that here we are using the fact that $(p_{n})_{n \geqslant 2}$ forms an increasing sequence, and so $p \geqslant p_{n}$ implies that $p \geqslant p_{s}$ for all $s \leqslant n$.  We are also using the fact that $g_{1}(\theta) = 1$ for all $\theta \in [0,1]$.  We proceed by summing (\ref{inequality: f-inequality2}) over all non-empty proper subsets of $[n]$, and then applying (\ref{inequality: f-inequality1}) to obtain
\begin{eqnarray}\label{inequality: f-inequality3}
(2^{n} - 2)\mu(C) \geqslant \Bigg(\sum_{A \subseteq [n]} g_{|A|}(\theta)g_{n-|A|}(\theta)\Bigg) - 2g_{0}(\theta)g_{n}(\theta) \nonumber \\ 
-  \Bigg(\sum_{A \subseteq [n]} \mu(Y_{A})\Bigg) + \mu(Y_{\emptyset})+ \mu(Y_{[n]}).
\end{eqnarray}
We apply (\ref{equation: f-equation}) and the fact that the events $C,Y_{\emptyset}$ and $Y_{[n]}$ are all the same event to (\ref{inequality: f-inequality3}) to get
\begin{eqnarray}\label{inequality: f-inequality4}
(2^{n} - 4)\mu(C) \geqslant (2^{n}-4)g_{n}(\theta) +2 -  \Bigg(\sum_{A \subseteq [n]} \mu(Y_{A})\Bigg).
\end{eqnarray}
Note that for all $A \subseteq [n]$, the events $Y_{A}$ and $Y_{A^{c}}$ are the same event, and so $\sum_{A \subseteq [n]}   \mu(Y_{A}) = 2\sum_{1 \in A \subseteq [n]}   \mu(Y_{A})$.  Moreover, the set $\{ Y_{A} : 1 \in A \subseteq [n]\}$ consists of pairwise disjoint events, and so $\sum_{1 \in A \subseteq [n]}   \mu(Y_{A}) \leqslant 1$.  Thus 
\begin{equation}\label{inequality: Y-inequality}
\sum_{A \subseteq [n]}\mu(Y_{A}) \leqslant 2.
\end{equation}

We apply (\ref{inequality: Y-inequality}) to (\ref{inequality: f-inequality4}) to obtain $(2^{n}-4)\mu(C) \geqslant (2^{n}-4)g_{n}(\theta)$.  As $n >2$, we have that $\mu(C) \geqslant  g_{n}(\theta)$ and so we are done.
\end{proof}

\subsection{A remark on $f_{k, K_n}(p)$ for $k\geqslant 2$}
Clearly we can define $f_{k,G}(p)$ analogously to $f_{1,G}(p)$ for $k \in \mathbb{N}_0$. For $k=0$, $f_{0, K_n}(p)$ is exactly the probability that an instance of the Erd{\H o}s--R\'enyi random graph $\mathbf{G}_{n,p}$ contains a spanning tree. As far as we know, there is no nice closed form expression for this function.

In this section, we have computed $f_{1, K_n}(p)$ exactly, which is the other interesting case, as for $k\geqslant 2$ the connectivity problem is trivial.
\begin{proposition}\label{prop: f_{k, K_n}(p) for k at least 2}
For all $k,n \in \mathbb{N}_{\geqslant 2}$, we have that 
\[f_{k, K_n}(p)=\begin{cases}
0 & \textrm{if }p\leqslant 1 -\frac{2}{n},\\
 1-\frac{n(1-p)}{2}& \textrm{otherwise.}
\end{cases}\]
\end{proposition}
\begin{proof}
For the lower bound, consider $\mu\in \mathcal{M}_{k,\geqslant p}(K_n)$. Since any subgraph of $K_n$ with at least $\binom{n}{2}-(n-1)$ edges is connected, we can apply Markov's inequality to show that
\begin{align*}
1-\mu(\{\textrm{connected}\})&\leqslant \mu (\{\exists \ \geqslant(n-1) \textrm{ closed edges}\})
\leqslant \frac{1}{n-1}\mathbb{E}_{\mu} \{\# \textrm{ closed edges}\}=\frac{n(1-p)}{2}.
\end{align*}
\noindent For the upper bound, consider the random graph $\mathbf{G}$ obtained as follows. Let $x:=\frac{1-p}{2}$. With probability $\min(nx, 1)$, select a vertex $i\in [n]=V(K_n)$ uniformly at random, and let $\mathbf{G}$ be the subgraph of $K_n$ obtained by removing all edges incident with $i$. Otherwise, let $\mathbf{G}$ be the complete graph $K_n$.  It is easy to check that $\mathbf{G}$ is a $2$-independent model with edge-probability $p$ and that $\mathbf{G}$ is connected if and only if $\mathbf{G}=K_n$, an event which occurs with probability $1-\min(1,nx)=\max\left(0, 1-n(1-p)/2\right)$.
\end{proof} 

\section{Cycles}\label{section: cycles}
\subsection{Linear programming for calculating $f_{1,G}(p)$}\label{subsection: linear programming}
In this subsection we describe how we can represent the problem of finding $f_{1,G}(p)$, for any graph $G$, as a (possibly non-linear) programme.

Given a graph $G$ on vertex set $[n]$, let $\mathcal{H}=\mathcal{H}(G)$ be the set of all labelled subgraphs of $G$.  Throughout this section we treat these subgraphs as subsets of $E(G)$, and always imagine them to be on the full vertex set $[n]$.  For each labelled subgraph of $G$ we write 
\begin{align*}
\mu (S) := \mu(S \subseteq \mathbf{G}_{\mu}) && \text{and} && \mu (\hat{S}) :=  \mu(\mathbf{G}_{\mu} = S). 
\end{align*}
Recall that for a function $\mu:\mathcal{H}\rightarrow \mathbb{R}_{\geqslant 0}$, we have $\mu \in \mathcal{M}_{1,\geqslant p}(G)$ if and only if the following three conditions all hold:
\begin{enumerate}
\item{$\mu$ is a probability measure on labelled subgraphs of $G$,}
\item{Every edge of $G$ is open in $\mathbf{G}_{\mu}$ with probability at least $p$,}
\item{Given non-empty $S,T \in \mathcal{H}$ such that $S$ and $T$ are supported on disjoint subsets of $[n]$,  $\mu(S) \cdot \mu(T)= \mu(S \cup T)$.}
\end{enumerate}
As we are interested in determining $f_{1,G}(p)$, and as randomly deleting edges cannot increase the probability of being connected,  we may assume that in fact every edge of $G$ is open in $\mathbf{G}_{\mu}$ with probability exactly $p$ (by applying random sparsification as in Remark~\ref{remark: random sparsification} if necessary).We can thus rewrite the conditions above in the following way:
\begin{enumerate}
\item{ $\sum_{H \in \mathcal{H}} \mu(\hat{H}) = 1$},
\item{ For all edges $e \in E(G)$, we have that$\sum_{H \in \mathcal{H}} \mathbbm{1}(e\in H)\mu(\hat{H}) = p$},
\item For all non-empty $S,T \in \mathcal{H}$, such that $S$ and $T$ are supported on disjoint subsets of $[n]$, we have that
\begin{equation}\label{equation: third bullet product}
\sum_{H \in \mathcal{H}} \mu(\hat{H}) \Big( \mathbbm{1}\big( (S \cup T) \subseteq H \big)  - \mathbbm{1}\big(S \subseteq H\big)  \mu(T) \Big)  = 0.
\end{equation}
\end{enumerate}
Let $A=A(G)$ be a matrix which has columns indexed by $\mathcal{H}$, and a row for each piece of information given by one of the above conditions. That is:
\begin{enumerate}
\item We have a row for the empty set such that $A_{\emptyset,H}:= 1$.
\item{We have a row for $e \in G$; the entry $A_{e,H}:=\mathbbm{1}(e \in H)$;}
\item{We have a row for each pair $S,T \in \mathcal{H}\setminus \{\emptyset\}$ supported on disjoint subsets of $[n]$; the entry $A_{\{S,T\},H}:=\mathbbm{1}((S \cup T) \subseteq H) - \mu(T) \cdot \mathbbm{1}(S \subseteq H)$.}
\end{enumerate}
Let $\mathbf{q}=\mathbf{q}(G)$ be a vector with indexing the same as the rows of $A$; let $q_\emptyset:=1$, $q_e:=p$ for $e \in G$, and $q_{\{S,T\}}:=0$ for each pair $S,T \in \mathcal{H} \setminus \{\emptyset\}$ supported on disjoint subsets of $[n]$.  Then a vector $\mathbf{w}$, whose entries are indexed by $\mathcal{H}$, which satisfies $w_{H} \geqslant 0$ for all $H \in \mathcal{H}$, and also $A \mathbf{w} = \mathbf{q}$
corresponds precisely to a measure $\mu \in \mathcal{M}_{1,\geqslant p}(G)$. 

Let $\mathbf{c}$ be a vector indexed by $\mathcal{H}$ defined by $c_{H} := \mathbbm{1}(H \text{ is connected})$.  Just to make it clear, we say that $H \in \mathcal{H}$ is connected if it contains a spanning tree of $[n]$. Then for a given value of $p$ the vector $\mathbf{w}(p)$ satisfying $A \mathbf{w}(p) = \mathbf{q}$ corresponds to a measure $\mu \in \mathcal{M}_{1,\geqslant p}(G)$ such that $\mu (H \text{ is connected}) = f_{1,G}(p)$.

Observe that for any graph with five vertices or fewer, any partition of the graph into two parts has that one part must have at most two vertices in it.  In particular, if $G$ is a graph on $[5]$, and $S$ and $T$ are non-empty subgraphs of $G$ supported on disjoint subsets of $[5]$, then one of $S$ and $T$ must consist of precisely one edge of $G$.  By choosing $T$ to be this subgraph, we can always choose $S$ and $T$ for (\ref{equation: third bullet product}) so that $\mu(T)=p$. Thus for any choice of $p$, we can turn the problem of finding $f_{1,G}(p)$ into the following linear programme:
\begin{equation}\label{equation: lp minimum}
a^{*}=\min_{\mathbf{w}}  \mathbf{c}^{T}\mathbf{w} \quad \text{subject to } A\mathbf{w}=\mathbf{q}, \mathbf{w} \geqslant 0.
\end{equation}
(Note that for graphs with six or more vertices, one may find $S$ and $T$ such that $\mu(T)$ (in (\ref{equation: third bullet product})) is an unknown function of $p$, and thus the programme is not linear; for example, this indeed is the case for $C_6$.) 

The duality theorem states that the asymmetric dual problem has the same optimal solution $a^{*}$:
\begin{equation}\label{equation: lp maximum}
a^{*}=\max_{\mathbf{x}}  \mathbf{q}^{T}\mathbf{x} \quad \text{subject to } A^{T}\mathbf{x} \leqslant \mathbf{c}.
\end{equation}
One can easily solve the linear programmes above for a specific value of $p$, for example using the software Maple, and the \textit{LPSolve} function it contains.  However we of course wish to find solutions for all values of $p \in [0,1]$.

By writing $A=(a_{ij})$, $\mathbf{w}=(w_j)$, $\mathbf{c}=(c_j)$, $\mathbf{q}=(q_i)$ and $\mathbf{x}=(x_i)$ any solutions $\mathbf{w}$ and $\mathbf{x}$ must satisfy $\sum_j a_{ij} w_j = q_i$, $\sum_i a_{ij} x_i \leqslant c_j$ and $w_i \geqslant 0$. Thus we have
$$\sum_{i} q_i x_i = \sum_i \left( \sum_j a_{ij} w_j \right) x_i = \sum_j \left( \sum_i a_{ij} x_i \right) w_j \leqslant \sum_j c_j w_j.$$
In particular for optimal solutions we have $\sum_{i} q_i x_i = \sum_{j} c_j w_j$ and so the inequality must be an equality, that is
$$ \left(\sum_i a_{ij} x_i \right) w_j = c_j w_j, \quad \text{ for all $j$.}$$
Consequently for each $j$ we either have $w_j=0$ or $\sum_i a_{ij} x_i=c_j$. Thus in our attempt to obtain a function for all $p$, it seems reasonable to look at an optimal solution for one value of $p$ and see which $w_j$ have been set to zero; assume for these indices that we always have $w_j=0$ and attempt to directly solve the equations that result from this. This motivates the following method:
\begin{itemize}
\item{Solve (\ref{equation: lp minimum}) with a specific value of $p$ to obtain a solution $\mathbf{w}(p)$ and a set $J:=\{ j \in [|\mathbf{w}|] : w_j(p)=0\}$.}
\item{Solve the set of equations $\{(A\mathbf{w})_i=q_i,w_j=0 : i \in [|\mathbf{w}|],j \in J\}$ to obtain functions of $p$ for all $w_k$, $k \in [|\mathbf{w}|]$, which we write as $w'_k(p)$.}
\item{Solve the set of equations $\{(A^T\mathbf{x})_i=c_i: i \in [|\mathbf{w}|] \setminus J\}$ to obtain functions of $p$ for all $x_k$,  $k \in [|\mathbf{w}|]$, which we write as $x'_k(p)$.}
\item{Write $w^*(p):= \mathbf{c}^{T}\mathbf{w'}(p)$ and $x^*(p):=\mathbf{q}^T\mathbf{x'}(p)$.}
\item{For a certain interval $P \subseteq [0,1]$ of values of $p$, check that $(A^T\mathbf{x}')_i(p) \leqslant c_i$ and $w'_i(p) \geqslant 0$, for all $i \in [|\mathbf{w}|]$.}
\end{itemize}
For the given interval $P$ which works above, the conditions above ensure that the $\mathbf{w}'(p)$ and $\mathbf{x}'(p)$ obtained in this way are feasible solutions to (\ref{equation: lp minimum}) and (\ref{equation: lp maximum}) respectively. Thus if $w^*(p)=x^*(p)$, then by the duality theorem we have $f_{1,G}(p)=w^*(p)$. Furthermore, a measure $\mu$ on the subgraphs of $G$ which is extremal is given directly by $\mathbf{w}'(p)$.  In the following subsection we give, as examples, two results which are proved using the above method.

\subsection{The connectivity function of small cycles}
In this subsection we prove Theorems~\ref{theorem: connected function C4} and~\ref{theorem: connected function C5} using the above method. Furthermore, the method gives us an extremal example in each case.
\begin{proof}[Proof of Theorem~\ref{theorem: connected function C4}]
For $C_4$ and $p \in [\frac{1}{2},1]$ an extremal construction is given by the measure $\mu$, defined by 
\begin{equation}
\mu(\hat{H}) = \begin{cases} 
2p-1 & \text{if } H=C_4;\\
\frac{p(1-p)}{2} & \text{if $H$ is contains precisely two edges, which are adjacent;}\\
(1-p)^2 & \text{if $H$ is contains precisely two edges, which are not adjacent;}\\
0 & \text{otherwise.}\\
\end{cases}
\nonumber
\end{equation}
For $C_4$ and $p \in [0,\frac{1}{2}]$ an extremal construction is given by the measure $\mu$, defined by 
\begin{equation}
\mu(\hat{H}) = \begin{cases} 
1-2p & \text{if $H$ is the empty graph;}\\
\frac{p(1-p)}{2} & \text{if $H$ is contains precisely two edges, which are adjacent;}\\
p^2 & \text{if $H$ is contains precisely two edges, which are not adjacent;}\\
0 & \text{otherwise.}\\
\end{cases}
\nonumber
\end{equation}
\end{proof}
We can in fact give a direct combinatorial proof of the lower bound in Theorem~\ref{theorem: connected function C4}: for any $\mu \in \mathcal{M}_{1,\geqslant p}(C_4)$, we have by $1$-independence that
\begin{align*}
\mu(\{\textrm{connected}\})&\geqslant \mu(\{12, 34 \textrm{ are open}\})-\mu(\{23, 14 \textrm{ are closed}\})\geqslant p^2-(1-p)^2=2p-1.
\end{align*}
Together with the first of the constructions of measures $\mu$ above (which can be found by analysing how the bound in the inequality above can be tight), this gives a second and perhaps more insightful proof of Theorem~\ref{theorem: connected function C4} than the one obtained from applying the linear optimisation method. However for the next result, on $f_{1,C_5}(p)$, we do not have a combinatorial proof, and our result relies solely on linear optimisation.
\begin{proof}[Proof of Theorem~\ref{theorem: connected function C5}]
For $C_5$ and $p \in [\frac{\sqrt{3}}{3},1]$ an extremal construction is given by the measure $\mu$, defined by 
\begin{equation}
\mu(\hat{H}) = \begin{cases} 
\frac{p(3p^2-1)}{3p-1} & \text{if } H=C_5;\\
\frac{p(1-p)(2p-1)}{5(3p-1)} & \text{if $H$ is missing precisely two edges, which are adjacent;}\\
\frac{p(1-p)^2}{5(3p-1)} & \text{if $H$ is missing precisely two edges, which are not adjacent;}\\
\frac{(2p-1)(1-p)^2}{3p-1} & \text{if $H$ is the empty graph;}\\
0 & \text{otherwise.}\\
\end{cases}
\nonumber
\end{equation}
For $C_5$ and $p \in [0,\frac{\sqrt{3}}{3}]$ an extremal construction is given by the measure $\mu$, defined by
\begin{equation}
\mu(\hat{H}) = \begin{cases} 
\frac{5p^3-5p^2-2p+2}{3p+2} & \text{if $H$ is the empty graph;}\\
\frac{p(1-3p^2)}{3p+2} & \text{if $H$ consists of precisely two edges, which are adjacent;}\\
\frac{p^3}{3p+2} & \text{if $H$ is missing precisely two edges, which are adjacent;}\\
\frac{p^2(1+p)}{3p+2} & \text{if $H$ is missing precisely two edges, which are not adjacent;}\\
0 & \text{otherwise.}\\
\end{cases}
\nonumber
\end{equation}
\end{proof}

\subsection{General bounds for cycles of length at least {$6$}}
We can use Markov's inequality to derive the following simple lower bound on $f_{1,C_n}(p)$ for $n\geqslant 6$.
\begin{proposition}\label{proposition: Markov for Ck}
	For $n \in \mathbb{N}$, with $n \geqslant 6$, and $p \in [0,1]$, we have $f_{1,C_n}(p) \geqslant \frac{np-(n-2)}{2}$.
\end{proposition}
A small adjustment to this argument gives the following improvement for $n=6$.
\begin{proposition}\label{proposition: improved Markov for C6}
	For $p \in [0,1]$ we have that $f_{1,C_6}(p) \geqslant  -p^3+3p^2-1$.
\end{proposition}
\begin{proof}[Proof of Proposition~\ref{proposition: Markov for Ck}]
Let $\mu \in \mathcal{M}_{1,\geqslant p}(C_k)$. Note that $\mathbf{G}_{\mu}$ is connected if and only if it has at most one closed edge. Thus by Markov's inequality, we have
\begin{eqnarray}
f_{1,C_k}(p) & = & 1-\mu (\exists \geqslant 2 \text{ closed edges in $C_k$}) \nonumber \\
& \geqslant & 1- \frac{\mathbb{E}_{\mu}(\#\text{ closed edges in $C_k$})}{2} \nonumber \\
& = & 1 - \frac{k(1-p)}{2} = \frac{kp-(k-2)}{2}. \nonumber
\end{eqnarray}
\end{proof}

\begin{proof}[Proof of Proposition~\ref{proposition: improved Markov for C6}]
Let $X$ be the number of closed edges in $\mathbf{G}_{\mu}$. 
Cyclically label the edges of $C_6$ as $e_1,\dots,e_6$. Then by simple counting, 
\begin{align*}
2(1-p)^3 & = \mu(e_1, e_3, e_5 \text{ are closed}) + \mu(e_2, e_4, e_6 \text{ are closed}) \\
& \leqslant \mu(X=3) + \mu(X=4) + \mu(X=5) + 2\mu(X=6).
\end{align*}
Now by simple counting again, linearity of expectation and the inequality above, we get: 
\begin{align*}
f_{1,C_6}(p) & = 1- \mu(X \geqslant 2) = 1 - \frac{\mathbb{E}_{\mu}(X)}{2}+ \\
& \left(\frac{\mu(X=1)+\mu(X=3)+2\mu(X=4)+3\mu(X=5)+4\mu(X=6)}{2}\right) \\
& \geqslant 1-\frac{6(1-p)}{2}+(1-p)^3=-p^3+3p^2-1.
\end{align*}
\end{proof}
\section{Maximising connectivity}\label{section: maximising connectivity}
In this section, we derive our results for maximising connectivity in $1$-independent modes. First of all Theorem~\ref{theorem: connected function Kn} allow us to easily determine the value of $F_{1, K_n(p)}$ and hence prove Theorem~\ref{theorem: maximising connectivity in Kn}.
\begin{proof}[Proof of Theorem~\ref{theorem: maximising connectivity in Kn}]
	Given a $1$-independent model $\mathbf{G}$ on $K_n$ with edge-probability at least $1-p$, observe that the complement $\mathbf{G}^c$ of $\mathbf{G}$ in $K_n$ is a $1$-independent model in which every edge is open with probability at most $p$. Furthermore, $\mathbf{G}^c$ is connected whenever $\mathbf{G}$ fails to be connected. This immediately implies
	\begin{align}\label{eq: complement of slow connectivity}
	1-f_{1, K_n}(1-p)\leqslant F_{1, K_n}(p).
	\end{align}
	Furthermore, observe that the Red-Blue measure $\nu_p$ we constructed to obtain the upper bound on $f_{1, K_n}(p)$ in the proof of Theorem~\ref{theorem: connected function Kn} has the property that a $\nu_{p}$-random graph is connected if and only if its complement fails to be connected. This immediately implies that we have equality in (\ref{eq: complement of slow connectivity}).
\end{proof}
\noindent For paths, a simple construction 
achieves the obvious upper bound for $F_{1,P_n}(p)$.
\begin{proof}[Proof of Theorem~\ref{theorem: maximising connectivity in Pn}]
	For any measure $\mu\in \mathcal{M}_{1, \leqslant p}(P_n)$, we have by $1$-independence that
	\begin{align*}
	\mu\left(\{\textrm{connected}\}\right)=\mu(\{P_n\})\leqslant \mu\left(\bigcap_{1\leqslant i \leqslant \lfloor \frac{n}{2}\rfloor}\{\textrm{the edge $\{2i-1, 2i\}$ is open}\}\right)\leqslant p^{\lfloor \frac{n}{2}\rfloor},
	\end{align*}
	which implies $F_{1,P_n}(p)\leqslant p^{\lfloor \frac{n}{2}\rfloor}$. For the lower bound, we construct a $1$-ipm as follows. For each integer $i$: $1\leqslant i \leqslant n/2$, we assign a state On to the vertex $2i$ with probability $p$, and a state Off otherwise, independently at random. Then set an edge of $P_n$ to be open if one of its endpoints is in state On, and closed otherwise. This is easily seen to yield a $1$-ipm $\mu$ on $P_n$ in which every edge is open with probability $p$, and for which
	\begin{align*}
	\mu\left(\{\textrm{connected}\}\right)=\mu\left(\bigcap_{1\leqslant i \leqslant \lfloor \frac{n}{2}\rfloor}\{\textrm{the vertex $2i$ is in state On}\}\right)= p^{\lfloor \frac{n}{2}\rfloor}.
	\end{align*}
	Thus $F_{1,P_n}(p)\geqslant p^{\lfloor \frac{n}{2}\rfloor}$, as claimed.
\end{proof}	
The case of cycles $C_n$ appears to be slightly more subtle. For the $4$-cycle, as in the previous section, we can give two proofs, one combinatorial and the other via linear optimisation.
\begin{proof}[Proof of Theorem~\ref{theorem: connected function C4}]
The theorem immediately follows from an application of the linear optimisation techniques from Section~\ref{section: cycles}. Alternatively, we can obtain the upper bound by a direct argument.	For any measure $\mu\in \mathcal{M}_{1, \leqslant p}(C_4)$, we have by $1$-independence that
	\begin{align*}
	1-\mu(\{\textrm{connected} \})&\geqslant \mu\left(\{\textrm{both $12$ and $34$ are closed}\}\right)\geqslant (1-p)^2,
	\end{align*}
	and, by a simple union bound and $1$-independence,
	\begin{align*}
	\mu(\{\textrm{connected} \})&\leqslant \mu\left(\{\textrm{both  $12$ and $34$ are open}\}\cup \{\textrm{both $23$ and $14$ are open}\}\right) \leqslant 2p^2.
	\end{align*}	
	Combining these two inequalities and using $1-(1-p)^2=2p-p^2$, we obtain
	\begin{align*}
	\mu(\{\textrm{connected} \})&\leqslant \min\left(2p^2, 2p-p^2\right),
	\end{align*}	
	which gives the claimed upper bound on $F_{1, C_4}(p)$.

	For the lower bound, we give two different constructions, depending on the value of $p$. For $p \in [\frac{2}{3},1]$ consider the measure $\mu$ defined by 
	\begin{equation}
	\mu(\hat{H}) = \begin{cases} 
	p(3p-2) & \text{if } H=C_4;\\
	p(1-p) & \text{if $H$ contains precisely three edges;}\\
	1-p(2-p) & \text{if $H$ is the empty graph with no edges;}\\
	0 & \text{otherwise.}\\
	\end{cases}
	\nonumber
	\end{equation}
	It is easily checked that $\mu \in \mathcal{M}_{1,\leqslant p}(C_4)$ and that $\mu(\{\textrm{connected}\})=1-(1-p(2-p))=2p-p^2$, which is maximal for $p$ in that range.

	For  $p \in [0,\frac{2}{3}]$, consider the measure $\mu$ defined by
	\begin{equation}
	\mu(\hat{H}) = \begin{cases} 
	\frac{p^2}{2} & \text{if $H$ contains precisely three edges;}\\
	\frac{p(2-3p)}{2} & \text{if $H$ contains precisely one edge;}\\
	1-4p(1-p) & \text{if $H$ is the empty graph with no edges;}\\
	0 & \text{otherwise.}\\
	\end{cases}
	\nonumber
	\end{equation}
	Again, it is easily checked that $\mu \in \mathcal{M}_{1,\leqslant p}(C_4)$ and that $\mu(\{\textrm{connected}\})=\mu\left(\{\geqslant 3\textrm{ edges open}\}\right)=2p^2$, which is maximal for $p$ in that range.	
\end{proof}

\begin{proof}[Proof of Theorem~\ref{theorem: connected function C5}]
	We simply apply the linear optimisation method from Section~\ref{section: cycles} --- here again we do not have a combinatorial proof. In addition to establishing the theorem, this gives us constructions of extremal $1$-independent measures maximising connectivity.

	For  $p \in [\frac{3}{5},1]$ an extremal construction is given by the measure $\mu$, defined by 
		\begin{equation}
	\mu(\hat{H}) = \begin{cases} 
	\frac{p(5p-3)}{5-3p} & \text{if $H=C_5$;}\\ 
	\frac{p(1-p^2)}{5-3p} & \text{if $H$ contains precisely four edges;}\\
	\frac{3p^3-7p^2+5p-1}{5-3p} & \text{if $H$ contains precisely two edges, which are adjacent;}\\
	\frac{2(1-p)^3}{5-3p} & \text{if $H$ contains precisely one edge;}\\
	0 & \text{otherwise.}\\ 
	\end{cases}
	\nonumber
	\end{equation}
	For  $p \in [\frac{1}{2}, \frac{3}{5}]$ an extremal construction is given by the measure $\mu$, defined by 
			\begin{equation}
	\mu(\hat{H}) = \begin{cases} 
	\frac{p^2}{3} & \text{if $H$ contains precisely four edges;}\\
	\frac{p(2-3p)}{3} & \text{if $H$ contains precisely two edges, which are adjacent;}\\
	\frac{p(2p-1)}{3} & \text{if $H$ contains precisely one edge;}\\
	\frac{3-5p}{3} & \text{if $H$ is the empty graph with no edges;}\\
	0 & \text{otherwise.}\\
	\end{cases}
	\nonumber
	\end{equation}
	For  $p \in [0,\frac{1}{2}]$ an extremal construction is given by the measure $\mu$, defined by 
			\begin{equation}
	\mu(\hat{H}) = \begin{cases} 
	\frac{p^2(p+1)}{p+4} & \text{if $H$ contains precisely four edges;}\\
	\frac{p^2(1-2p)}{p+4} & \text{if $H$ is missing precisely two edges, which are adjacent;}\\
	\frac{p(p^2-3p+2)}{p+4} & \text{if $H$ contains precisely two edges, which are adjacent;}\\
	\frac{5p^2-9p+4}{p+4} & \text{if $H$ is the empty graph with no edges;}\\
	0 & \text{otherwise.}\\
	\end{cases}
	\nonumber
	\end{equation}
\end{proof}

\section{Proof of Theorem~\ref{theorem: upper bounds on long paths critical prob}}\label{section: long paths}
Combining Corollary~\ref{corollary: bound on long path constant in products with the line} with our results on $1$-independent connectivity, much of Theorem~\ref{theorem: upper bounds on long paths critical prob} is immediate.
\begin{proof}[Proof of Theorem~\ref{theorem: upper bounds on long paths critical prob}]
	For the lower bound in part (i), we note that $\mathbb{Z}\times C_n$ has the finite 2-percolation property.  Thus, as described after the proof of Theorem \ref{theorem: local construction integer lattice}, we have that $p_{1, \ell\mathit{p}}(\mathbb{Z}\times C_n)\geqslant 4-2 \sqrt{3}$.  For the upper bound in part (i), since the long paths critical probability is non-decreasing under the addition of edges, we have 
	\[p_{1, \ell\mathit{p}}(\mathbb{Z}\times C_n)\leqslant p_{1, \ell\mathit{p}}(\mathbb{Z}\times P_n) \leqslant p_{1, \ell\mathit{p}}(\mathbb{Z}\times P_2),\]
	which is at most $2/3$ by Theorem~\ref{theorem: long paths critical prob}(ii).

	For the upper bounds (ii)--(iv) Theorem~\ref{theorem: upper bounds on long paths critical prob} follow directly from our results on $1$-independent connectivity functions. For $G=K_3, C_4, C_5$, we plug in the value of $f_{1, G}(p)$ in equation~(\ref{equation:  pstar}), solve for $p_{\star}(G)$ and apply Corollary~\ref{corollary: bound on long path constant in products with the line}.

	In part (v), we begin by noting that as we are considering an increasing nested sequence of graphs, the sequence $\left(p_{1, \ell \mathit{p}}(\mathbb{Z}\times K_n)\right)_{n\in \mathbb{N}}$ is non-increasing in $[0,1]$ and hence tends to a limit as $n\rightarrow \infty$. For the lower bound in (v), observe that for any $n\in \mathbb{N}$ the graph $\mathbb{Z}\times K_n$ has the finite 2-percolation property -- indeed for any finite $k$, the closure of a copy of $P_k\times K_n$ under $2$-neighbour bootstrap percolation in $\mathbb{Z}\times K_n$ is equal to itself. We construct a $1$-ipm $\mu$ on $\mathbb{Z}\times K_n$ as in Corollary~\ref{corollary: if 2-percolation property, then lb on p_{1,c}} but with starting set $T_0=\{0\}\times V(K_n)$ and hence $T_k=(\{k\}\times V(K_n))\cup (\{-k\}\times V(K_n))$.  It is easily checked that $\mu$-almost surely, all components (and hence all paths) in a $\mu$-random graph have length at most $5n$. Since by construction $d(\mu)=4-2\sqrt{3}$, this proves 
	\[p_{1, \ell \mathit{p}}\left(\mathbb{Z}\times K_n\right)\geqslant 4-2\sqrt{3}\]
	for all $n \in \mathbb{N}$. For the upper bound, we perform some simple analysis. By solving a quadratic equation, we see that
	\[\left(\frac{1+\sqrt{2p-1}}{2}\right)^2 >  (1-p)\]
	for all fixed $p\in (\frac{5}{9}, 1)$. Then by Theorem~\ref{theorem: connected function Kn}, for any such fixed $p$ and all $n$ sufficiently large, we have that 
	\begin{align*}
\left(f_{1, K_n}(p)\right)^2&> \left(\frac{1+\sqrt{2p-1}}{2}\right)^{2n}>4(1-p)^n= 4(1-p)^{v(K_n)}.
	\end{align*}
Thus $p_{\star}(K_n)<p$ for all $n$ sufficiently large, which by Corollary~\ref{corollary: bound on long path constant in products with the line} implies $p_{1, \ell \mathit{p}}\left(\mathbb{Z}\times K_n\right)<p$.
\end{proof}

\section{Open problems}\label{section: open problems}
\subsection{More tractable subclasses of $1$-independent measures}\label{subsection: problems on simpler classes of 1-ipms}
\noindent The most obvious open problem about $1$-independent percolation is of course whether the known lower and upper bounds on $p_{1, c}(\mathbb{Z}^2)$ can be improved. This problem is, we suspect, very hard in general. However, it may prove more tractable if we restrict our attention to a smaller family of measures.
\begin{definition}
	Let $G$ be a graph. A $G$-partition is a partitioned set $\sqcup_{v\in V(G)}\Omega_v$, with non-empty parts indexed by the vertices of $G$. A $G$-partite graph is a graph $H$ on a $G$-partition $V(H)=\sqcup_{v\in V(G)}\Omega_v$ whose edges are a subset of the union of the complete bipartite graphs $\sqcup_{uv\in E(G)}\{\omega_u\omega_v: \ \omega_u \in \Omega_u, \omega_v \in\Omega_v\}$ corresponding to the edges of $G$.
\end{definition}
\noindent Given a $G$-partite graph $H$ on a $G$-partition $\sqcup_{v\in V(G)}\Omega_v$, we have a natural way of constructing $1$-independent bond percolation models: given a family $\mathbf{X}=\left(S_v\right)_{v\in V(G)}$ of independent random variables with $S_v$ taking values in $\Omega_v$, the \emph{$(H, \mathbf{X})$-random subgraph} of $G$, denoted by $H[\mathbf{X}]$, is the random configuration on $E(G)$ obtained by setting $uv$ to be open if and only if $S_uS_v\in E(H)$.
\begin{definition}\label{def: state based measure}
	Let $G$ be a graph. A measure $\mu \in \mathcal{M}_{1,\geqslant p}(G)$ is said to be vertex-based if there exist
	\begin{itemize}
		\item a $G$-partition $\sqcup_{v\in V(G)}\Omega_v$,
		\item  an associated $G$-partite graph $H$, and
		\item a collection of independent random variables $(S_v)_{v\in V(G)}$ with $S_v$ taking values in $\Omega_v$,
	\end{itemize} 
	such that the $(H, \mathbf{X})$-random subgraph $H[\mathbf{X}]$ has the same distribution as the $\mu$-random graph $\mathbf{G}_{\mu}$.
\end{definition}
\noindent Let $\mathcal{M}_{\mathrm{vb}, \geqslant p}(G)$ denote the collection of all vertex-based measures on $G$ with edge-probability at least $p$.
\begin{problem}\label{problem: state-based measures for percolation}
	Determine $\inf\Bigl\{p\in [0,1]: \ \forall \mu \in \mathcal{M}_{\mathrm{vb}, \geqslant p}(\mathbb{Z}^2), \ \mu(\{\mathrm{percolation}\})=1 \Bigr\}$.
\end{problem}
\noindent 
Vertex-based measures arise naturally in renormalising arguments, and are thus a natural class of examples to consider. A special case of Problems~\ref{problem: Harris critical probabiliy for 1-ipm} and~\ref{problem: state-based measures for percolation} is obtained by further restricting our attention to the case where the $\Omega_v$ have bounded size.
\begin{definition}
	A vertex-based measure $\mu$	on a graph $G$ is $N$-uniformly bounded if it as in Definition~\ref{def: state based measure} above and in addition for each $v\in V(G)$, $\vert \Omega_v\vert \leqslant N$. Furthermore, a vertex-based measure $\mu$ on a graph $G$ is uniformly bounded if it is $N$-uniformly bounded for some $N\in \mathbb{N}$.
\end{definition}
Let $\mathcal{M}_{N-\mathrm{ubvb}, \geqslant p}(G)$ and $\mathcal{M}_{\mathrm{ubvb}, \geqslant p}(G)$  denote the collection of all vertex-based measures on $G$ with edge-probability at least $p$ that are $N$-uniformly bounded and uniformly bounded respectively.
\begin{problem}\label{problem: ub state-based measures for percolation}
\begin{enumerate}[(i)]
	\item For $N\in \mathbb{N}$, determine 
	\[\inf\Bigl\{p\in [0,1]: \ \forall \mu \in \mathcal{M}_{N-\mathrm{ubvb}, \geqslant p}(\mathbb{Z}^2), \ \mu(\{\mathrm{percolation}\})=1 \Bigr\}.\]
	\item  Determine 
	\[\inf\Bigl\{p\in [0,1]: \ \forall \mu \in \mathcal{M}_{\mathrm{ubvb}, \geqslant p}(\mathbb{Z}^2), \ \mu(\{\mathrm{percolation}\})=1 \Bigr\}.\]
\end{enumerate} 
\end{problem}
Finally, let us note that the second most obvious problem arising from our work, besides that of improving the bounds on $p_{1, c}(\mathbb{Z}^2)$, is arguably that of giving bounds on $p_{1, \ell \mathit{p}}(\mathbb{Z}^2)$ and closely related variants. Such problems, which correspond to new questions in extremal graph theory, are discussed in the subsections below. For these problems too we believe restrictions to the class of uniformly bounded vertex-based $1$-ipms could be both fruitful and interesting in their own right.
\subsection{Harris critical probability for other lattices}
\noindent Beyond $\mathbb{Z}^2$, it is natural to ask about bounds on $p_{1, c}(G)$ for some of the other commonly studied lattices in percolation theory.
\begin{problem}\label{problem: 1-indep crit prob for other lattices}
	Give good bounds on the value of $p_{1,c}(G)$ when $G$ is one of the eleven Archimedean lattices in the plane or the $d$-dimensional integer lattice $\mathbb{Z}^d$.
\end{problem}
\noindent This problem is particularly interesting when $G$ is the triangular lattice or the honeycomb lattice (two lattices for which the $0$-independent Harris critical probability is known exactly), or the cubic integer lattice $\mathbb{Z}^3$ (which is important in applications). A challenge in all cases is finding constructions of non-percolating $1$-independent measures with high edge-probability --- indeed, our arsenal of constructions for $1$-independent percolation problems is so sparse that any new construction could be of independent interest.

In a different direction, we can observe that $\mathbb{Z}^{d+1}$ contains a copy of $\mathbb{Z}^d$, whence the sequence $\left(p_{1, c}(\mathbb{Z}^d)\right)_{d\in \mathbb{N}}$ is non-increasing in $[0,1]$ and converges to a limit. Balister and Bollob\'as asked for its value:
\begin{problem}\label{problem: limit of 1-ip Harris crit prob for Zd}[Balister, Bollob\'as~\cite{BalisterBollobas12}]
	Determine
	$\lim_{d\rightarrow \infty} p_{1,c}(\mathbb{Z}^d)$.
\end{problem}
\noindent Note that by Theorem~\ref{theorem: local construction integer lattice} proved in this paper, this limit must be at least $4-2\sqrt{3}$.
\subsection{Other notions of $1$-independent critical probabilities}\label{subsection: different critical probabilities}
Let $G$ be an infinite, locally finite connected graph, and $v_0$ a fixed vertex of $G$. Given a bond percolation model $\mu$ on $G$, we let $\mathbf{C}_{v_0}$ denote the connected component of $\mathbf{G}_{\mu}$ containing $v_0$.

If $\mu$ is $0$-independent, then $\mu(\{\mathrm{percolation}\})=1$ if and only if $\mu(\{\vert \mathbf{C}_0\vert =\infty\})>0$. However this need not be true for a $1$-independent measure. Indeed, consider the $1$-ipm on $\mathbb{Z}^2$ obtained by taking the measure constructed in the proof of Theorem~\ref{theorem: local construction integer lattice} to determine the state of the edges in the $\ell_{\infty}$ ball of radius $3$ around the origin and setting every other edge to be open independently at random with probability $4-2\sqrt{3}$. Then in this model percolation occurs almost surely, but the origin is contained inside a component of order at most $28$.

Thus in principle there are different edge-probability thresholds in $1$-independent percolation on a graph $G$ for percolation to occur \emph{somewhere} with probability $1$ and for it to occur \emph{anywhere} with strictly positive probability. Indeed, if $p_{1,c}(\mathbb{Z}^2)$ were strictly less than $3/4$, then one could obtain examples of such a graph $G$  by attaching a long path to the origin in $\mathbb{Z}^2$.

Another critical edge-probability of interest is the \emph{Temperley critical probability}, which in $0$-independent percolation is the threshold $p_T$ at which $\mathbb{E}\vert \mathbf{C}_v\vert =\infty$ for any vertex $v$ (and every $0$-independent measure with edge-probability $>p_T$). In general this threshold is different from the Harris critical probability. Again for $1$-independent percolation we have that the threshold for \emph{some} vertex $v\in V(G)$ to satisfy $\mathbb{E}\vert \mathbf{C}_v\vert =\infty$ and for the threshold for \emph{all} vertices of $G$ to satisfy this are different.

\begin{problem}\label{problem: different critical prob}
	Given an infinite, locally finite connected graph $G$, determine the following four critical probabilities:
	\begin{align*}
	p_{1,T_1}(G)&:=\inf\Bigl\{p\in [0,1]: \ \forall \mu \in \mathcal{M}_{1,\geqslant p}(G), \exists v\in V(G): \ \mathbb{E}_{\mu} \vert \mathbf{C}_v\vert =\infty \Bigr\},\\
	p_{1,T_2}(G)&:=\inf\Bigl\{p\in [0,1]: \ \forall \mu \in \mathcal{M}_{1,\geqslant p}(G), \forall v\in V(G): \ \mathbb{E}_{\mu} \vert \mathbf{C}_v\vert =\infty \Bigr\},\\
	p_{1,H_1}(G)&:=\inf\Bigl\{p\in [0,1]: \ \forall \mu \in \mathcal{M}_{1,\geqslant p}(G), \exists v\in V(G): \ \mu\left(\vert \mathbf{C}_v\vert =\infty\right)>0 \Bigr\},\\
		p_{1,H_2}(G)&:=\inf\Bigl\{p\in [0,1]: \ \forall \mu \in \mathcal{M}_{1,\geqslant p}(G), \forall v\in V(G): \ \mu\left(\vert \mathbf{C}_v\vert =\infty\right)>0 \Bigr\}.
	\end{align*}
\end{problem}
 It follows from their definition that these four critical probabilities satisfy
\begin{align}\label{eq: ineq between four crit prob}
p_{1,T_1}(G)\leqslant p_{1,T_2}(G) \leqslant p_{1, H_2}(G) \quad \textrm{ and } \quad p_{1,T_1}(G)\leqslant p_{1,H_1}(G) \leqslant p_{1, H_2}(G).
\end{align}
In general, these four critical probabilities are all different. Indeed, Balister and Bollob\'as showed in~\cite[Theorem 1.5]{BalisterBollobas12} that there exists an infinite, locally finite connected graph $G$ with $p_{1, H_1}(G)=\frac{1}{2}$.  For any $p$: $\frac{1}{2}<p<\frac{3}{4}$, we have shown in Theorem~\ref{theorem: connected function Pn} that there exists $N$ such that $f_{1, P_N}(p)=0$. Attach one end of a path of length $N$ to an arbitrary vertex of $G$ to form a graph $G_1$, and let $v$ denote the other end of the path. Then there exist $1$-ipm $\mu \in \mathcal{M}_{1, \geq p}(G_1)$ such that with probability $1$ the component of $v$ in a $\mu$-random graph has order at most $N$, which is finite. Thus we have 
\[p_{1,T_1}(G_1)\leq p_{1, H_1}(G_1)\leq p_{1, H_1}(G)=\frac{1}{2}<p\leq p_{1,T_2}(G_1) \leq p_{1, H_2}(G_1).\]
On the other hand consider a graph $G_2$ obtained from the line lattice by attaching to each vertex $i\in \mathbb{Z}$ a collection of  $2^{\vert i \vert +2}$ leaves.  Clearly, $p_{1, H_1}(G_2)= p_{1, H_1}(\mathbb{Z})=1$. Now consider a $1$-ipm $\mu\in \mathcal{M}_{1, \geq \frac{3}{4}}(G_2)$. 
By Theorem~\ref{theorem: crossing G times Pn} applied to $G=K_1$ and $\alpha=\frac{1}{2}$, for any path $P$ of length $i$ in $G_2$, the $\mu$-probability that all edges in $P$ are open is at least $2^{-(i+1)}$. Thus for any $v_0\in V(G_2)$, the  expected size of $\vert \mathbf{C}_{v_0}\vert$ is
\begin{align*}\mathbb{E}_{\mu}\vert \mathbf{C}_{v_0}\vert =\sum_{v\in V(G_2)}\mu\left(\{v\in \mathbf{C}_{v_0}\}\right)&\geq \sum_{i \in \mathbb{Z}_{\geq 2}} \#\{v: \ \textrm{ the path from $v_0$ to $v$ has length $i$}\} 2^{-(i+1)}\\&\geq \sum_{i\in \mathbb{Z}_{\geq 2}} 2^{i+1} 2^{-(i+1)}=\infty.\end{align*}
Thus we have
\[p_{1, T_1}(G_2)\leq p_{1, T_2}(G_2)\leq \frac{3}{4}<1=p_{1,H_1}(G_2)=p_{1, H_2}(G_2).\]
\begin{corollary}\label{corollary: optimal inequalities between four crit prob}
\begin{enumerate}[(i)]
	\item None of the inequalities in (\ref{eq: ineq between four crit prob}) may be replaced by an equality. 
	\item Neither $p_{1, T_2}(G)\leqslant p_{1, H_1}(G)$ nor the reverse inequality are true in general. \qedsymbol
\end{enumerate}	
\end{corollary}
Observe that $p_{1, H_1}(G)$ is the $1$-independent Harris critical probability $p_{1,c}(G)$ studied in this paper; given Corollary~\ref{corollary: optimal inequalities between four crit prob}, we more precisely should call it the \emph{first} Harris critical probability for $1$-independent percolation. Our construction for the proof of Theorem~\ref{theorem: local construction integer lattice} and the argument of Balister, Bollob\'as and Walters from~\cite{BalisterBollobasWalters05} give the following bounds when $G=\mathbb{Z}^2$:
 \[ 4-2\sqrt{3} \leqslant p_{1,T_1}(\mathbb{Z}^2) \leqslant p_{1,H_2}(\mathbb{Z}^2)\leqslant 0.8639. \]
\begin{question}\label{question: equality of the critical prob in Z^2}
	Are any of the four critical probabilities from Problem~\ref{problem: different critical prob} equal when $G=\mathbb{Z}^2$?
\end{question}
Finally, note that Problem~\ref{problem: different critical prob} asks, in essence, how much we can delay percolation phenomena relative to the $0$-independent case by exploiting the local dependencies between the edges allowed by $1$-independence. While perhaps less useful in applications, it is an equally natural and appealing extremal problem to ask how much we can use these local dependencies to instead hasten the emergence of an infinite connected component. Balister and Bollob\'as were the first to consider this problem in~\cite{BalisterBollobas12}, which it would be remiss not to mention here.
\begin{definition}\label{def: fast percolation}
	Let $G$ be an infinite, locally finite connected graph, and let $\mathcal{M}_{k, \leqslant p}(G)$  be as before the collection of $k$-ipms $\mu$ on $G$ satisfying $\sup_{e\in E(G)}\mu\{e\textrm{ is open}\}\leqslant p$. The critical threshold for fast $k$-independent percolation on $G$ is
	\[p_{k,F}(G) :=\inf\Bigl\{p\in[0,1]: \ \exists \mu \in \mathcal{M}_{k,\leqslant p}(G): \ \mu(\{\mathrm{percolation}\})=1\Bigr\}.\]
\end{definition}
Balister and Bollob\'as determined $p_{1,F}(G)$ when $G$ is an infinite, locally finite tree, and also gave the simple general bounds
\begin{align}\label{eq: Balister Bollobas bounds on fast percolation}
		\frac{1}
		{\left(\mu_{\mathrm{conn.}}(G)\right)^2 } \leqslant p_{1,F}(G)\leqslant \left(\theta_{\mathrm{site}}(G)\right)^2,
\end{align}
where $\mu_{\mathrm{conn.}}(G)$ is the connective constant of $G$ and $\theta_{\mathrm{site}}(G)$ the critical value of the $\theta$-parameter for site percolation on $G$. For the square integer lattice, this gives a lower bound on $p_{1,F}(\mathbb{Z}^2)$ of $0.1393$ from known upper bounds on $\mu_{\mathrm{conn.}}(\mathbb{Z}^2)$. In the other direction, we get a rigorous upper bound of  $p_{1,F}(\mathbb{Z}^2)$ of $0.4618$ and non-rigorous upper bound of $0.3515$ from bounds and estimates for $\theta_{\mathrm{site}}(\mathbb{Z}^2)$. This obviously leaves a big gap, which Balister and Bollob\'as asked to reduce.
\begin{question}\label{question: fast per threshold}[Balister and Bollob\'as~\cite{BalisterBollobas12}]
What is $p_{1, F}(\mathbb{Z}^2)$?	
\end{question}

\subsection{Long paths critical probability}\label{subsection: problems on long paths}
An obvious problem is to tighten the bounds in Theorem~\ref{theorem: upper bounds on long paths critical prob}(v), which are not too far apart (compared to many of the other bounds on critical probabilities for $1$-independent model).
\begin{problem}\label{problem: improving lp crit constant bounds for Kn times long paths}
Determine $\lim_{n\rightarrow \infty} p_{1, \ell \mathit{p}}\left(\mathbb{Z}\times K_n\right)$ (which must be an element of $[4-2\sqrt{3}, \frac{5}{9}]$).	
\end{problem}	
\noindent In a similar vein, the sequence $p_{1, \ell \mathit{p}}(\mathbb{Z}\times P_n)$ is a non-increasing function of $n$ (since $\mathbb{Z}\times P_{n+1}$ contains $\mathbb{Z}\times P_n$ as a subgraph). In this paper, we have given constructions showing that for all integers $n\geqslant 3$,
\[4-2\sqrt{3} \leqslant p_{1, \ell\mathit{p}}(\mathbb{Z}\times C_n)\leqslant p_{1, \ell\mathit{p}}(\mathbb{Z}\times P_n)
\leqslant \frac{2}{3}= p_{1, \ell \mathit{p}}(\mathbb{Z}\times P_2).\]
Thus the sequence $\left(p_{1, \ell \mathit{p}}(\mathbb{Z}\times P_n)\right)_{n \in \mathbb{N}}$ tends to a limit in the interval $[4-2\sqrt{3}, \frac{2}{3}]$ as $n\rightarrow \infty$.
\begin{problem}\label{problem: limit of longs paths prob on Z times Pn as n tends to infinity}
	Determine 
	\[p_{1, \ell \mathit{p}} \left(\mathbb{Z}\times P_{\infty}\right):=\lim_{n\rightarrow \infty} p_{1, \ell \mathit{p}}(\mathbb{Z}\times P_{n}).\]
\end{problem}
\noindent An in principle different but related problem is determining the value of the long paths critical probability in $\mathbb{Z}^2$ (which need not be equal to the quantity $p_{1, \ell \mathit{p}} \left(\mathbb{Z}\times P_{\infty}\right)$ defined above).
\begin{problem}\label{problem: limit of longs paths prob on Z times Z}
	Determine $p_{1, \ell \mathit{p}} \left(\mathbb{Z}^2\right)$.
\end{problem}

We can also ask for $k$-independent versions of the long paths critical probability. Defining $p_{k, \ell \mathit{p}}\left(G\right)$ mutatis mutandis,  
it is straightforward to adapt our arguments and constructions from Section~\ref{section: line lattice} to show the following result, which also follows directly from the work of Liggett, Schonmann and Stacey~\cite{LiggettSchonmannStacey97} on stochastic domination of $0$-independent measures on $\mathbb{Z}$ by $k$-independent ones.
\begin{theorem}\label{theorem: k-independent long paths crit prob}[Liggett, Schonman and Stacey~\cite{LiggettSchonmannStacey97}]
	For any $k\in \mathbb{N}_{0}$, we have
	\[p_{k, \ell \mathit{p}}(\mathbb{Z})=1-\frac{k^k}{(k+1)^{k+1}},\]
	with the convention that $0^0=1$.
\end{theorem}
Given $k$ fixed, it is easy to construct a $3k$-ipm $\mu$ on $\mathbb{Z}^2$ with $d(\mu)=1-\frac{2}{k}$ and no open path of length more than $(2k+1)^2$. Indeed,  build a random graph model as follows:
\begin{itemize}
	\item begin with all edges of $\mathbb{Z}^2$ open;
	\item independently for each $(i,j)\in \mathbb{Z}^2$, choose $H_{ij}\in [k+1]$ uniformly at random and then for all $j'$: $j(k+1)- k \leq j' \leq j(k+1)+k$, set the horizontal edge $\{(i(k+1)  +H_{ij}-1,   j'), (i(k+1)+H_{ij}, j' )$ to  be closed;
	\item independently for each $(i,j)\in \mathbb{Z}^2$, choose $V_{ij}\in [k+1]$ uniformly at random and then for all $i'$: $i(k+1)- k \leq i' \leq i(k+1)+k$, set the vertical edge  $\{(i',   j(k+1)+V_{ij}-1), (i', j(k+1)+V_{ij})$ to  be closed.
\end{itemize}
It is easy to check that this random graph model is $(3k-1)$-independent, has edge probability at least $1-\frac{2}{k}+\frac{1}{k^2}$ and that every connected component has order at most $(2k+1)^2$. 
\begin{corollary}\label{corollary: k-limit of long paths critical probability is 1}
	For any fixed $k\in \mathbb{N}$,
	\[p_{(3k-1), \ell \mathit{p}} (\mathbb{Z}^2)\geq 1-\frac{2}{k}+\frac{1}{k^2}. \qquad \qed\]
\end{corollary}
\noindent In particular we have $\lim_{k\rightarrow \infty}p_{k, \ell \mathit{p}} (\mathbb{Z}^2)=1$ (and in fact a similar construction shows this remains true in $\mathbb{Z}^d$).


Finally, as in Section~\ref{subsection: different critical probabilities}, we should observe that the almost sure existence of arbitrarily long open paths in a $1$-independent model on $G$ does not imply that for every $\ell\in \mathbb{N}$ every vertex of $G$ has a strictly positive probability of being part of a path of length at least $\ell$. Thus we may actually define a second long paths critical probability,
\[p_{1, \ell\mathit{p}_2}(G):=\inf\Bigl\{p\in [0,1]: \ \forall \mu \in \mathcal{M}_{1,\geqslant p}(G), \forall v\in V(G), \forall \ell\in \mathbb{N},\  \mu(\exists \textrm{ open path from $v$ of length }\ell )>0 \Bigr\}.\]
\begin{problem}
	Determine $p_{1, \ell\mathit{p}_2}(\mathbb{Z}^2)$.
\end{problem}
\noindent Our construction in the proof of Theorem~\ref{theorem: local construction integer lattice} shows that $p_{1, \ell\mathit{p}_2}(\mathbb{Z}^2)\geqslant 4-2\sqrt{3}$, and we know it is upper-bounded by $p_{1, H_2} (\mathbb{Z}^2)\leqslant 0.8639$. As in Section~\ref{subsection: problems on simpler classes of 1-ipms}, it may be fruitful to study the long paths critical constant when one restricts one's attention to a smaller class of $1$-ipms. In particular, by considering the class of uniformly bounded vertex-based measures, one is led to the following intriguing problem in graph theory.

Given an $n$-uniformly bounded $\mathbb{Z}^2$-partite graph $H$ with partition $\sqcup_{v\in \mathbb{Z}^2} \Omega_v$. A \emph{transversal subgraph} of $H$ is a subgraph of $H$ induced by a set of distinct representatives $S$ for the parts of $H$, i.e. a set of vertices of $H$ such that $\vert S\cap \Omega_v\vert=1$ for all $v\in \mathbb{Z}^2$. The \emph{$G$-partite density} of $H$ is
\[d_G(H):=\inf \Bigl\{ \frac{ e(H[\Omega_u \sqcup \Omega_v])}{\vert \Omega_u\vert \cdot \vert \Omega_v \vert}: \ uv\in E(\mathbb{Z}^2) \Bigr\} .\]
\begin{question}\label{question: graph theory}
Suppose $H$ is an $n$-uniformly bounded $\mathbb{Z}^2$-partite graph in which in every transversal subgraph the connected component containing the origin is...
\begin{enumerate}[(a)]
	\item ... of size at most $C$, for some constant $C\in \mathbb{N}$.
	\item ... finite.
\end{enumerate}
How large can $d_G(H)$ be?	
\end{question}	
\noindent This question can be viewed as a problem from extremal multipartite graph theory. Plausibly some tools from that area, in particular the work of Bondy, Shen, Thomass\'e and Thomassen~\cite{BondyShenThomasseThomassen06} and Pfender~\cite{Pfender12}, could be brought to bear on it.

\subsection{Connectivity function}
We determined in Sections~\ref{section: cycles} and~\ref{section: maximising connectivity} the connectivity function $f_{1,C_n}(p)$  for cycles $C_n$ of length at most $5$. It is natural to ask what happens for longer cycles.
\begin{problem}\label{problem: connectivity of cycles}
	Determine $f_{1,C_n}(p)$ for $n \in \mathbb{N}_{\geqslant 6}$.	
\end{problem}	
As mentioned in Section~\ref{section: cycles}, the problem of finding $f_{1,C_6}(p)$ is non-linear. Nevertheless, one can use software, such as Maple and its contained NLPSolve function, to try to estimate the answer. This suggests the following:
\begin{itemize}
	\item{The threshold at which $f_{1,C_6}(p)$ becomes nonzero is approximately $p=0.59733$;}
	\item{For $p$ just above this threshold, the best `asymmetric' (see the next subsection for a definition) measure is \emph{better} than the best `symmetric' measure; e.g. at $p=0.62$ we have $f_{1,C_6}(0.62)$ is approximately $0.007$, but is as high as $0.11$ when restricted to `symmetric' measures.}
\end{itemize}
More generally, one can ask what happens in cycles if we have higher dependency or if we try to maximise connectivity rather than minimise.
\begin{problem}\label{problem: minimising connectivity for cyles}
	Determine $f_{k, C_n}(p)$ for all $p\in [0,1]$, $k\in \mathbb{N}$ and integers $n\geqslant k+2$.
\end{problem}
\begin{problem}\label{problem: maximising connectivity for cyles}
	Determine $F_{k, C_n}(p)$ for all $p\in [0,1]$, $k\in \mathbb{N}$ and integers $n\geqslant k+2$.
\end{problem}
Beyond paths, cycles and complete graphs, the $1$-independent connectivity problem is perhaps most natural to study in the hypercube graph $Q_n$ and in the $n\times n$ toroidal grid $C_n\times C_n$. Progress on either of these would likely lead to progress on other problems in $1$-independent percolation as well.
\begin{problem}\label{problem: connectivity of hypercubes}
	Determine $f_{1,Q_n}(p)$ for all $n\geqslant 3$.
\end{problem}	
\begin{problem}\label{problem: connectivity of toroidal grids}
	Determine $f_{1,C_n\times C_n}(p)$ for all $n \geqslant 3$.
\end{problem}

In a different direction, we can ask whether the extremal measures attaining $f_{1, G}(p)$ can be required to have `nice' properties. For $C_4$ and $p \in [0,1/2]$ another extremal construction for $f_{1,C_4}(p)$ is given by the measure $\mu$, defined by
\begin{equation}
\mu(\hat{H}) = \begin{cases} 
1-2p & \text{if $H$ is the empty graph;}\\
p(1-p) & \text{if $H$ is $\{12,14\}$ or $\{23,34\}$};\\
p^2 & \text{if $H$ is contains precisely two edges, which are not adjacent;}\\
0 & \text{otherwise.}\\
\end{cases}
\nonumber
\end{equation}
Motivated by the above, we call a measure $\mu \in \mathcal{M}_{1,\geqslant p}(G)$ \emph{symmetric} if for any pair of labelled subgraphs $S$ and $T$ of $G$ such that there exists an automorphism of $G$ mapping $S$ to $T$, then $\mu(\hat{S})=\mu(\hat{T})$. Note that the above measure is an example of a non-symmetric extremal construction for $f_{1,C_4}(p)$, whereas the measure given at the end of Section~\ref{subsection: linear programming} is symmetric. This leads to the following question.
\begin{question}\label{question: symmetry}
	For any $G$ and any $p \in [0,1]$, does there always exist a symmetric measure $\mu \in \mathcal{M}_{1,\geqslant p}(G)$ which achieves $f_{1,G}(p)$?
\end{question}
If the programme for solving $f_{1,G}(p)$ attained via our method in Section~\ref{subsection: linear programming} is linear, then the answer for $G$ is yes (see the appendix for a proof of this fact).


Another natural question is when the extremal connectivity can be attained by vertex-based measures.
\begin{question}\label{question: state-based}
	For which $G$ and which $p$ does there exist a vertex-based measure $\mu \in \mathcal{M}_{1,\geqslant p}(G)$ which achieves $f_{1,G}(p)$?
\end{question}

\section*{Acknowledgements}
First and foremost, the second author wants to express his debt of gratitude to Mark Walters, who introduced him to $1$-independent percolation during his PhD thesis, and with whom he constructed the $1$-independent measure on the ladder given in Section~\ref{section: ladder}.

Secondly the authors would like to collectively thank Paul Balister, Alexander Holroyd, Oliver Riordan and Amites Sarkar for stimulating discussions, encouragement and useful references and advice.

The work in this paper was undertaken during the Spring semester 2018, while the third author was a guest researcher in the Ume{\aa} under the auspices of the Ume{\aa}--Birmingham Erasmus exchange agreement. We are grateful for the Erasmus support that made this visit possible. Research of the first and second author was funded by the grant VR 2016-03488 from the Swedish Research Council, whose support is also gratefully acknowledged.

Finally, the authors would like to thank the two anonymous referees for their careful work on the paper, and for their helpful comments which led to improvements in the exposition.

\section*{Appendix}\label{appendix: symmetric proof}
Let $S_n$ be the symmetric group on $n$ elements. For a graph $G$ let $\Gamma(G)$ be the automorphism group of $G$, 
$\Gamma(G):=\{ \sigma \in S_n: \sigma(i) \sigma(j) \in E(G) \text{ if and only if } ij \in E(G),  \text{ for all }i,j \in [n]\}$. Enumerate the elements of $\Gamma(G)$ as $\sigma_i$, $i \in [|\Gamma(G)|]$, and for $H \subseteq G$ write $H_{\sigma_i}$ for its image under $\sigma_i$.
Recall $\mu$ is \emph{symmetric} if for all $H \subseteq G$ and for all $i,j$ we have $\mu(\hat{H}_{\sigma_i})=\mu(\hat{H}_{\sigma_j})$.

\begin{theorem}
Let $G$ be a graph with vertex set $[n]$ such that the optimisation problem for $f_{1,G}(p)$ is linear. Suppose that $\mu$ is a non-symmetric measure which achieves $f_{1,G}(p)$ for some value of $p$. Then for this same value of $p$ there exists another $\mu'$ which is symmetric and also achieves $f_{1,G}(p)$.
\end{theorem}

\begin{proof}
Let $\mu$ be the non-symmetric measure which achieves $f_{1,G}(p)$. For all $H \subseteq G$ define 
\begin{equation}
\mu'(\hat{H}):=\frac{1}{|\Gamma(G)|} \sum_{j=1}^{|\Gamma(G)|} \mu(\hat{H}_{\sigma_j}).
\nonumber
\end{equation}
First note that we have the following:
\begin{eqnarray}\label{mu1}
\mu'(S_{\sigma_j}) & =&  \sum_{H \subseteq G} \mathbbm{1}(S_{\sigma_j} \subseteq H) \mu'(\hat{H})  \nonumber \\
 &=& \frac{1}{|\Gamma(G)|}  \sum_{H \subseteq G} \sum_{i=1}^{|\Gamma(G)|}  \mathbbm{1}(S_{\sigma_j} \subseteq H_{\sigma_i})  \mu'(\hat{H}) \nonumber \\
& =& \frac{1}{|\Gamma(G)|}  \sum_{H \subseteq G} \sum_{i=1}^{|\Gamma(G)|}  \mathbbm{1}(S_{\sigma_i} \subseteq H)  \mu'(\hat{H}) \nonumber \\
 &= & \frac{1}{|\Gamma(G)|} \sum_{i=1}^{|\Gamma(G)|} \mu(S_{\sigma_i}).
\end{eqnarray}
The first and final equalities follow by definition. The second equality follows by summing through each automorphism of $H$ and the fact that $\mu'(\hat{H}_{\sigma_i})=\mu'(\hat{H}_{\sigma_j})$ for all $i,j$. The third equality follows by swapping automorphisms of $H$ to automorphisms of $S$, which again works since $\mu'(\hat{H}_{\sigma_i})=\mu'(\hat{H}_{\sigma_j})$. 

Now note that if $S$ is the empty graph or a single edge, then $\mu(S_{\sigma_i})=\mu(S_{\sigma_j})$ for all $i,j$ and thus we obtain $\mu'(S_{\sigma_i})=\mu(S_{\sigma_i})$ for all $i$. It easily follows that $\mu'$ is a measure with edge-probability $p$. We must show $\mu'(S) \cdot \mu'(T) = \mu'( S \cup T)$ for all $S,T$ which are labelled non-empty subgraphs of $G$  supported on disjoint subsets of vertices. If the optimisation problem is linear, then without loss of generality we have $\mu'(T)=p$, and so this follows by linearity and (\ref{mu1}). It remains to show that $\mu'$ also achieves $f_{G}(p)$. Again this follows easily since 
\begin{eqnarray}
 \mu' (\text{a $\mu'$-random graph is connected})  &=& \sum_{H \subseteq G} \mathbbm{1}(\text{$H$ is connected}) \mu'(H) \nonumber \\
 &= & \frac{1}{|\Gamma(G)|}  \sum_{H \subseteq G} \sum_{i=1}^{|\Gamma(G)|}  \mathbbm{1}(\text{$H_{\sigma_i}$ is connected}) \mu'(H_{\sigma_i}) \nonumber\\
&= & \mu (\text{a $\mu$-random graph is connected}),\nonumber
\end{eqnarray}
where the second equality follows by summing through each automorphism of $H$, and the third since $H_{\sigma_i}$ is connected if and only if $H_{\sigma_j}$ is connected, for all $i,j$.
\end{proof}
\end{document}